\documentclass[a4paper,11pt,leqno]{article}
\usepackage{amsmath,amssymb,amsthm}
\usepackage{graphicx}
\usepackage{longtable}
\usepackage{stmaryrd}
\usepackage{tikz}
\usetikzlibrary{calc}

\newtheorem{theo}{Theorem}[section]
\newtheorem{lemm}[theo]{Lemma}

\newtheorem{prop}[theo]{Proposition}

\newtheorem{rem}{Remark}[section]

\numberwithin{equation}{section}

\newsavebox{\boxI}
\sbox{\boxI}{\begin{tikzpicture}
\coordinate (A1) at (0,0);
\coordinate (A2) at ($(0,0.2)$);
\fill (A2) circle (1pt);
\draw (A1)--(A2);
\end{tikzpicture}}
\newcommand{\I}{{\usebox{\boxI}}}

\newsavebox{\boxY}
\sbox{\boxY}{\begin{tikzpicture}
\coordinate (A1) at (0,0);
\coordinate (A2) at ($0.6*(0.1,0.2)$);
\coordinate (A3) at ($0.6*(-0.1,0.2)$);
\coordinate (A4) at ($0.6*(0,-0.2)$);
\foreach \n in {2, 3} \fill (A\n) circle (0.9pt);
\draw (A3)--(A1)--(A2);
\draw (A1)--(A4);
\end{tikzpicture}}
\newcommand{\Y}{{\usebox{\boxY}}}

\newsavebox{\boxV}
\sbox{\boxV}{\begin{tikzpicture}
\coordinate (A1) at (0,0);
\coordinate (A2) at ($0.9*(0.1,0.2)$);
\coordinate (A3) at ($0.9*(-0.1,0.2)$);
\foreach \n in {2, 3} \fill (A\n) circle (1pt);
\draw (A3)--(A1)--(A2);
\end{tikzpicture}}
\newcommand{\V}{{\usebox{\boxV}}}

\newsavebox{\boxW}
\sbox{\boxW}{\begin{tikzpicture}
\coordinate (A1) at (0,0);
\coordinate (A2) at ($0.5*(-0.1,0.2)$);
\coordinate (A3) at ($0.5*(0.1,0.2)$);
\coordinate (A4) at ($0.5*(-0.2,0.4)$);
\coordinate (A5) at ($0.5*(0,0.4)$);
\coordinate (A6) at ($0.5*(0,-0.2)$);
\foreach \n in {3,4,5} \fill (A\n) circle (0.8pt);
\draw (A2)--(A1)--(A3);
\draw (A4)--(A2)--(A5);
\draw (A6)--(A1);
\end{tikzpicture}}
\newcommand{\W}{{\usebox{\boxW}}}

\newsavebox{\boxWc}
\sbox{\boxWc}{\begin{tikzpicture}
\coordinate (A1) at (0,0);
\coordinate (A2) at ($0.5*(-0.1,0.2)$);
\coordinate (A3) at ($0.5*(0.1,0.2)$);
\coordinate (A4) at ($0.5*(-0.2,0.4)$);
\coordinate (A5) at ($0.5*(0,0.4)$);
\foreach \n in {3,4,5} \fill (A\n) circle (0.9pt);
\draw (A2)--(A1)--(A3);
\draw (A4)--(A2)--(A5);
\draw (A3) to [out=90,in=0] (A5);
\end{tikzpicture}}
\newcommand{\Wc}{{\usebox{\boxWc}}}

\newsavebox{\boxIWc}
\sbox{\boxIWc}{\begin{tikzpicture}
\coordinate (A1) at (0,0);
\coordinate (A2) at ($0.5*(-0.1,0.2)$);
\coordinate (A3) at ($0.5*(0.1,0.2)$);
\coordinate (A4) at ($0.5*(-0.2,0.4)$);
\coordinate (A5) at ($0.5*(0,0.4)$);
\coordinate (A6) at ($0.5*(0,-0.2)$);
\foreach \n in {3,4,5} \fill (A\n) circle (0.9pt);
\draw (A2)--(A1)--(A3);
\draw (A4)--(A2)--(A5);
\draw (A6)--(A1);
\draw (A3) to [out=90,in=0] (A5);
\end{tikzpicture}}
\newcommand{\IWc}{{\usebox{\boxIWc}}}

\newsavebox{\boxK}
\sbox{\boxK}{\begin{tikzpicture}
\coordinate (A1) at (0,0);
\coordinate (A2) at ($1*(-0.1,0.1)$);
\coordinate (A3) at ($1*(0,0.2)$);
\fill (A3) circle (1pt);
\draw (A1)--(A2)--(A3);
\end{tikzpicture}}
\newcommand{\K}{{\usebox{\boxK}}}

\newsavebox{\boxB}
\sbox{\boxB}{\begin{tikzpicture}
\coordinate (A1) at (0,0);
\coordinate (A2) at ($0.6*(-0.2,0.2)$);
\coordinate (A3) at ($0.6*(0.2,0.2)$);
\coordinate (A4) at ($0.6*(-0.3,0.4)$);
\coordinate (A5) at ($0.6*(-0.1,0.4)$);
\coordinate (A6) at ($0.6*(0.1,0.4)$);
\coordinate (A7) at ($0.6*(0.3,0.4)$);
\foreach \n in {4,5,6,7} \fill (A\n) circle (0.9pt);
\draw (A2)--(A1)--(A3);
\draw (A4)--(A2)--(A5);
\draw (A6)--(A3)--(A7);
\end{tikzpicture}}
\newcommand{\Bo}{{\usebox{\boxB}}}

\newsavebox{\boxD}
\sbox{\boxD}{\begin{tikzpicture}
\coordinate (A1) at (0,0);
\coordinate (A2) at ($0.4*(-0.1,0.2)$);
\coordinate (A3) at ($0.4*(0.1,0.2)$);
\coordinate (A4) at ($0.4*(-0.2,0.4)$);
\coordinate (A5) at ($0.4*(0,0.4)$);
\coordinate (A6) at ($0.4*(-0.3,0.6)$);
\coordinate (A7) at ($0.4*(-0.1,0.6)$);
\foreach \n in {3,5,6,7} \fill (A\n) circle (0.7pt);
\draw (A2)--(A1)--(A3);
\draw (A4)--(A2)--(A5);
\draw (A6)--(A4)--(A7);
\draw (A1) [fill=white] circle (1pt);
\end{tikzpicture}}
\newcommand{\D}{{\usebox{\boxD}}}

\newsavebox{\boxC}
\sbox{\boxC}{\begin{tikzpicture}
\coordinate (A1) at (0,0);
\coordinate (A2) at ($1*(0.1,0.1)$);
\coordinate (A3) at ($1*(-0.1,0.1)$);
\coordinate (A4) at ($1*(0,0.2)$);
\foreach \n in {2, 4} \fill (A\n) circle (1pt);
\draw (A4)--(A3)--(A1)--(A2);
\draw (A1) [fill=white] circle (1pt);
\end{tikzpicture}}
\newcommand{\Co}{{\usebox{\boxC}}}

\renewcommand{\a}{\alpha}
\renewcommand{\b}{\beta}
\newcommand{\e}{\varepsilon}
\newcommand{\de}{\delta}
\newcommand{\fa}{\varphi}
\newcommand{\ga}{\gamma}
\renewcommand{\k}{\kappa}
\newcommand{\la}{\lambda}
\renewcommand{\th}{\theta}
\newcommand{\si}{\sigma}

\renewcommand{\t}{\tau}

\newcommand{\Ga}{\Gamma}

\def\${|\!|\!|}
\newcommand{\rs}{\varodot}
\newcommand{\pl}{\varolessthan}
\newcommand{\pr}{\varogreaterthan}

\def\R{{\mathbb{R}}}
\def\N{{\mathbb{N}}}
\def\Z{{\mathbb{Z}}}
\def\T{{\mathbb{T}}}

\allowdisplaybreaks

\def\bn{\mathbf{n}}
\def\bm{\mathbf{m}}
\newcommand{\MARU}[1]{{\ooalign{\hfil#1\/\hfil\crcr
     \raise.167ex\hbox{\mathhexbox20D}}}}

\begin{document}

\title{A coupled KPZ equation, its two types of approximations and existence of global solutions}
\author{Tadahisa Funaki and Masato Hoshino}
\date{\today}
\maketitle

\begin{abstract}
\noindent
This paper concerns the multi-component coupled Kardar-Parisi-Zhang (KPZ) equation and its two types of approximations.  One approximation is obtained as a simple replacement of the noise term by a smeared noise with a proper renormalization, while the other one introduced in \cite{F15} is suitable for studying the invariant measures.  
By applying the paracontrolled calculus introduced by Gubinelli et al.\ \cite{GIP, GP}, 
we show that two approximations have the common limit under the properly adjusted choice of
renormalization factors for each of these approximations.
In particular, if the coupling constants of the nonlinear term of the coupled KPZ equation satisfy the so-called \lq\lq trilinear"
condition, the renormalization factors can be taken the same in two approximations and the difference
of the limits of two approximations are explicitly computed.  Moreover, under the trilinear condition, the Wiener measure
twisted by the diffusion matrix becomes stationary for the limit and we show that the solution of
the limit equation exists globally in time when the initial value is sampled from the stationary measure.
This is shown for the associated tilt process.
Combined with the strong Feller property shown by Hairer and Mattingly \cite{HM16}, this result can be extended for all initial values.
\footnote{
\hskip -6mm
Graduate School of Mathematical Sciences,
The University of Tokyo, Komaba, Tokyo 153-8914, Japan.
e-mails: funaki@ms.u-tokyo.ac.jp, hoshino@ms.u-tokyo.ac.jp}
\footnote{
\hskip -6mm
\textit{Abbreviated title $($running head$)$: Coupled KPZ equation.}}
\footnote{
\hskip -7mm
\textit{
The first author is supported in part by the JSPS KAKENHI Grant Numbers $($S$)$ 24224004, $($S$)$ 16H06338, $($B$)$ 26287014 and 26610019. The second author is supported by JSPS KAKENHI, Grant-in-Aid for JSPS Fellows, 16J03010.
}}
\end{abstract}


\section{Introduction and main results}

\subsection{Coupled KPZ equation}

We consider the following $\R^d$-valued coupled KPZ equation for $h(t,x)=(h^\a(t,x))_{\a=1}^d$ defined on the one dimensional torus $\T \equiv \R/\Z = [0,1)$:
\begin{equation}  \label{3_eq:KPZ}
\partial_t h^\a = \tfrac12 \partial_x^2 h^\a + \tfrac12 \Ga_{\b\ga}^\a\partial_x h^\b \partial_x h^\ga + \si^\a_\b \xi^\b, \quad x \in \T,
\end{equation}
for $1 \le \a \le d$. Here summation symbols $\sum$ over $\b$ and $\gamma$ are omitted by Einstein's convention. $(\si_\b^\a)_{1\le \a, \b\le d}$ and $(\Ga_{\b\ga}^\a)_{1\le \a,\b,\ga\le d}$ are given constants, and $\xi(t,x)= (\xi^\a(t,x))_{\a=1}^d$ is an $\R^d$-valued space-time Gaussian white noise.  In particular, it has the covariance structure
$$
E[\xi^\a(t,x)\xi^\b(s,y)] = \de^{\a\b}\de(x-y)\de(t-s),
$$
where $\de^{\a\b}$ denotes Kronecker's $\de$. We always assume that the coupling
constants $\Ga_{\b\ga}^\a$ satisfy 
\begin{equation} \label{3_eq:1.2}
\Ga_{\b\ga}^\a= \Ga_{\ga\b}^\a
\end{equation}
for all $\a, \b, \ga$, and the diffusion matrix $\si=(\si_\b^\a)_{1\le \a, \b\le d}$ is invertible.
The symmetry or bilinearity \eqref{3_eq:1.2} of $\Ga^\a = (\Ga^\a_{\b\ga})_{\b\ga}$ for each $\a$ is natural
due to the form of the equation \eqref{3_eq:KPZ}.

One of the motivations to study the coupled KPZ equation \eqref{3_eq:KPZ} comes from the nonlinear fluctuating hydrodynamics recently discussed by Spohn and others \cite{FSS,S14,SS15}, whose origin goes back to Landau. From microscopic systems with random evolutions, in a proper space-time scaling, one can derive certain nonlinear partial differential equations (PDEs) as a result of  a local average due to the local ergodicity. This procedure is called the hydrodynamic limit. If the system has $d$ (local) conserved quantities, we have a system of $d$ coupled nonlinear PDEs in the limit. The noises in the microscopic systems are averaged out and disappear in the macroscopic limit equations. However, if we consider a linearization of this system around a global equilibrium, the noise terms survive in a proper scaling and we obtain linear stochastic PDEs (SPDEs) in the limit. At least heuristically, 
if the system involves a weak asymmetry and if we expand the equation to the second order, one can expect to obtain the coupled KPZ equations in the limit in a proper scaling. If some of $\Ga^\a_{\b\ga}$ are degenerate, then the solution involves different scalings
such as diffusive, KPZ or  (anomalous) L\'evy type scalings. 

\subsection{Two approximating equations}

The coupled KPZ equation \eqref{3_eq:KPZ} itself is ill-posed, so that we need to introduce its approximations;
see \cite{FQ} for a scalar-valued KPZ equation. A simple approximation of \eqref{3_eq:KPZ} is defined as follows. Let $\eta \in C_0^\infty(\R)$ be a function satisfying $\eta(x) = \eta(-x)$ and $\int_\R \eta(x)dx=1$; note that $\eta$ may not be non-negative.
We set $\eta^\e(x) = \eta(x/\e)/\e$ for $\e>0$ and consider the $\R^d$-valued KPZ approximating equation for $h= h^\e(t,x) \equiv (h^{\e,\a}(t,x))_{\a=1}^d$ with a smeared noise and a proper renormalization:
\begin{equation}  \label{3_eq:1-2}
\partial_t h^{\e,\a} = \tfrac12 \partial_x^2 h^{\e,\a} + \tfrac12 \Ga_{\b\ga}^\a
(\partial_x h^{\e,\b} \partial_x h^{\e,\ga} -c^\e A^{\b\ga}-B^{\e,\b\ga})
+ \si_\b^\a\xi^{\b}*\eta^\e,
\end{equation}
for $1 \le \a \le d$, where $A^{\b\ga}=\sum_{\de=1}^d \si_\de^\b \si_\de^\ga$, $c^\e = \frac1\e \|\eta\|_{L^2(\R)}^2$ and $B^{\e,\b\ga}$ is a renormalization factor defined in Section \ref{3_section 4}, which diverges as $O(\log\e^{-1})$ as $\e\downarrow 0$ in general. We consider $\e>0$ small enough, so that the support of $\eta^\e$ is in the interval $(-1/2,1/2)$.

Another approximation of \eqref{3_eq:KPZ} suitable for studying invariant measures is introduced as follows. Let $\eta_2(x) = \eta*\eta(x)$, $\eta_2^\e(x) =\eta_2(x/\e)/\e$ and consider the following $\R^d$-valued equation for $\tilde h= \tilde h^\e(t,x) \equiv (\tilde h^{\e,\a}(t,x))_{\a=1}^d$ with a smeared noise and a proper renormalization:
\begin{equation}  \label{3_eq:1}
\partial_t \tilde h^{\e,\a} = \tfrac12 \partial_x^2\tilde  h^{\e,\a} + \tfrac12 \Ga_{\b\ga}^\a (\partial_x \tilde h^{\e,\b} \partial_x \tilde h^{\e,\ga} -c^\e A^{\b\ga}-\tilde B^{\e,\b\ga})* \eta_2^\e
+ \si_\b^\a\xi^{\b}*\eta^\e,
\end{equation}
for $1 \le \a \le d$, where $\tilde B^{\e,\b\ga}$
is a renormalization factor defined in Section \ref{3_section 4}, which diverges as $O(\log\e^{-1})$ as $\e\downarrow 0$ in general. 
We assume that the support of $\eta_2^\e$ is in $(-1/2,1/2)$.
The difference of \eqref{3_eq:1} from \eqref{3_eq:1-2} is that it has a convolution factor $*\eta_2^\e$ in the nonlinear term.  

In \cite{F15}, assuming that $\si$ is an identity matrix $I$, under the additional assumption,
which we call the trilinear condition, on $\Ga$:
\begin{equation} \label{3_eq:1.2-s}
\Ga_{\b\ga}^\a=\Ga_{\ga\b}^\a= \Ga_{\ga\a}^\b,
\end{equation}
for all $\a,\b,\ga$, the infinitesimal invariance of the smeared Wiener measure for the tilt process
$u^\epsilon=\partial_x\tilde h^\epsilon$ of the solution $\tilde h^\epsilon$ of 
\eqref{3_eq:1} with $\tilde B^{\e,\b\ga}=0$ is shown (actually on $\R$ instead of $\T$). Namely, let
 $(B_x)_{x\in\T} = \big((B_x^\a)_{\a=1}^d\big)_{x\in\T}$ be the $d$-dimensional periodic Brownian motion 
such that $B_0 =B_1=0$ a.s. Then the distribution of $\partial_x(B*\eta^\e)$ is infinitesimally invariant
for $u^\epsilon=\partial_x\tilde h^\epsilon$ determined from \eqref{3_eq:1} with $\si=I$ and $\tilde B^{\e,\b\ga}=0$.

This result can be easily extended to our general setting with $\si$. Indeed, let $\tilde h^\e=(\tilde h^{\e,\a})$ be the solution of \eqref{3_eq:1} and set $\hat h^{\e,\a} := \t_\b^\a \tilde h^{\e,\b}$, where $\t = (\t_\b^{\a})$ is the inverse matrix of $\si$. Then, we easily see that $\hat h^\e = (\hat h^{\e,\a})$ is a solution of \eqref{3_eq:1} with $(\si_\b^\a\xi^\b*\eta^\e,A^{\b\ga}, \tilde B^{\e,\b\ga}, \Ga_{\b\ga}^\a)$ replaced by $(\xi^\a*\eta^\e,\de^{\b\ga},
\t_{\b'}^\b\t_{\ga'}^\ga\tilde B^{\e,\b'\ga'}, \hat\Ga_{\b\ga}^\a)$, where
\begin{equation}  \label{3_eq:1.hatGa}
\hat\Ga_{\b\ga}^\a = \t_{\a'}^\a \Ga_{\b'\ga'}^{\a'} \si_\b^{\b'}
\si_\ga^{\ga'},
\end{equation}
which arises from $\Ga$ under the change of variables $\hat h = \tau \tilde h$. In fact, $\Ga$ is a tensor of type (1,2) and $\hat{\Ga}$ defined by \eqref{3_eq:1.hatGa} is its transform under the change of basis.
Note that the bilinearity \eqref{3_eq:1.2}: $\hat\Ga_{\b\ga}^\a= \hat\Ga_{\ga\b}^\a$ automatically holds. Therefore, if $\hat\Ga$ determined from $\Ga$ as in \eqref{3_eq:1.hatGa} satisfies the trilinear condition:
\begin{equation} \label{3_eq:1.2-s-hat}
\hat\Ga_{\b\ga}^\a=\hat\Ga_{\ga\b}^\a= \hat\Ga_{\ga\a}^\b,
\end{equation}
for all $\a,\b,\ga$, then the distribution of the derivative of
the $d$-dimensional periodic and smeared Brownian motion $\big( \partial_x (\si B*\eta^\e)\big)_{x\in \T}
=\big((\partial_x \si_\b^\a B^\b*\eta^\e(x))_{\a=1}^d\big)_{x\in\T}$ multiplied by $\si$ is infinitesimally invariant 
for the tilt process $u=\partial_x\tilde h$ of the solution $\tilde h$ of \eqref{3_eq:1} with $\tilde B^{\e,\b\ga}=0$.

When $d=1$ and $\Ga_{\b\ga}^\a = \si_\b^\a =1$ for simplicity, the approximating equations \eqref{3_eq:1-2} with $B^{\e,\b\ga}=0$ and \eqref{3_eq:1} with $\tilde B^{\e,\b\ga}=0$ have the forms:
\begin{equation}  \label{3_eq:1-3}
\partial_t h = \tfrac12 \partial_x^2 h + \tfrac12 
\big((\partial_x h)^2 -c^\e \big)
+ \xi*\eta^\e,
\end{equation}
and
\begin{equation}  \label{3_eq:1-4}
\partial_t \tilde h = \tfrac12 \partial_x^2 \tilde h + \tfrac12 
\big((\partial_x \tilde h)^2 -c^\e \big)* \eta_2^\e
+ \xi*\eta^\e,
\end{equation}
respectively. It is shown that the solution of \eqref{3_eq:1-3} converges as $\e\downarrow 0$ to the so-called Cole-Hopf solution $h_{\text{CH}}(t,x)$ of the KPZ equation \cite{Hairer13,Hai14}, while the solution of \eqref{3_eq:1-4} converges to $h_{\text{CH}}(t,x) + \frac1{24}t$ under the equilibrium setting \cite{FQ} and the non-equilibrium setting for a maximal solution \cite{Ho}. The method of \cite{FQ} is based on the Cole-Hopf transform, which is not available for our multi-component coupled equation in general. 

\subsection{Main results}

Our first goal is to study the limits of the solutions of two types of approximating equations \eqref{3_eq:1-2} and \eqref{3_eq:1} as $\e\downarrow 0$ based on the paracontrolled calculus introduced by Gubinelli et al.\ \cite{GIP,GP} as in \cite{Ho} for $d=1$. Especially, we study the difference between these two limits, which extends the results for the scalar-valued
KPZ equation mentioned above.
For $\k\in\R$ and $r\in\N$, $(\mathcal{C}^\k)^r:=\mathcal{B}_{\infty,\infty}^\k(\T;\R^r)$ denotes the $\R^r$-valued 
Besov space on $\T$.  Our first two main theorems are formulated as follows.

\begin{theo}  \label{3_thm:1.1}
{\rm (1)} Let $0<\delta<\delta'<\frac12$ be fixed. For every $h(0)\in(\mathcal{C}^\de)^d$, there exists a unique solution $h^\e$ of the KPZ approximating equation \eqref{3_eq:1-2} up to the survival time $T_{\text{\rm sur}}^\e\in (0,\infty]$ (i.e. $T_{\text{\rm sur}}^\e=\infty$ or $\lim_{T\uparrow T_{\text{\rm sur}}^\e}\|h^\e\|_{C([0,T],(\mathcal{C}^\de)^d)}=\infty$). With a proper choice of $B^{\e,\b\ga}$, there exists a random time $T_{\text{\rm sur}}\in(0,\infty]$ such that $T_{\text{\rm sur}}\le\liminf_{\e\downarrow 0}T_{\text{\rm sur}}^\e$ in probability and $h^\e$ converges to some $h$ in $C([0,T],(\mathcal{C}^\delta)^d)\cap C((0,T],(\mathcal{C}^{\de'})^d)$ in probability for every $0<T< T_{\text{\rm sur}}$. This $T_{\text{\rm sur}}$ can be chosen maximal in the sense that $T_{\text{\rm sur}}=\infty$ or $\lim_{T\uparrow T_{\text{\rm sur}}}\|h\|_{C([0,T],(\mathcal{C}^\de)^d)}=\infty$.  The survival time $T_{\text{\rm sur}}$ depends on the initial value $h(0)$ and driving processes introduced in Section \ref{3_section3.2}. \\
{\rm (2)} A similar result holds for the solution $\tilde h^\e$ of the KPZ approximating equation
\eqref{3_eq:1} with some limit $\tilde h$ under a proper choice of $\tilde B^{\e,\b\ga}$. Moreover, under a well-adjusted choice of the renormalization factors $B^{\e,\b\ga}$ and  $\tilde B^{\e,\b\ga}$ as in Section \ref{3_section 4},
we can make $h=\tilde h$.
\end{theo}

\begin{rem}
Precisely, the convergence $h^\epsilon\to h$ considered here means that
\begin{align*}
&P(\|h^\epsilon-h\|_{C([0,T],(\mathcal{C}^\delta)^d)\cap C([t,T],(\mathcal{C}^{\de'})^d)}>\lambda,\ T<T_{\text{\rm sur}}\wedge T_{\text{\rm sur}}^\epsilon)\\
&+P(T_{\text{\rm sur}}^\epsilon\le T_{\text{\rm sur}}-\lambda,\ T_{\text{\rm sur}}<\infty)+P(T_{\text{\rm sur}}^\epsilon\le T,\ T_{\text{\rm sur}}=\infty)\to0
\end{align*}
for every $0<t<T$ and $\lambda>0$. The convergence $\tilde{h}^\epsilon\to h$ is similarly understood.
\end{rem}

\begin{theo}  \label{3_thm:1.2}
All components of the renormalization matrices $B^{\e}$ and $\tilde{B}^{\e}$ defined in Section \ref{3_section 4} behave as $O(1)$ if and only if the trilinear condition \eqref{3_eq:1.2-s-hat} holds. In particular, when \eqref{3_eq:1.2-s-hat} holds, we can choose $B^\e=\tilde B^\e=0$ in the approximating equations \eqref{3_eq:1-2} and \eqref{3_eq:1}, and the corresponding solutions $h_{B=0}^\e$ and $\tilde h_{\tilde B=0}^\e$ converge to $h_{B=0}$ and $\tilde h_{\tilde B=0}$, respectively, as $\e\downarrow 0$. In the limit, we have
$$
\tilde h_{\tilde B=0}^\a(t,x) = h_{B=0}^\a(t,x) + c^\a t, \quad 1 \le \a \le d,
$$
where
$$
c^\a =  \frac1{24} \sum_{\b_1,\b_2} \si_\b^\a \hat\Ga_{\a_1\a_2}^\b 
 \hat\Ga_{\b_1\b_2}^{\a_1} \hat\Ga_{\b_1\b_2}^{\a_2}.
$$
\end{theo}

\begin{rem}
For the equation \eqref{3_eq:1-2} with $d=1$ (then the condition \eqref{3_eq:1.2-s-hat} is trivial), Hairer \cite{Hairer13} first obtained that the two logarithmic renormalization factors (i.e., $O(\log\epsilon^{-1})$ terms) cancel with each other and the constant $\frac1{24}$ arises from the difference of these two terms, see also Section \ref{3_section 4}.
\end{rem}

\begin{rem}
Kupiainen and Marcozzi \cite{KpMr} studied another approximation of the equation \eqref{3_eq:KPZ} with $\sigma=I$ and obtained the cancellation of the logarithmic renormalization factors under the trilinear condition \eqref{3_eq:1.2-s}.
\end{rem}

Our second goal is to show the global-in-time existence of the limit process $h$ under the condition \eqref{3_eq:1.2-s-hat}. Let $\mu_A$ be the Gaussian measure on the space $(\mathcal{C}_0^{\delta-1})^d:=\{u\in(\mathcal{C}^{\delta-1})^d\,; \int_\T u=0\}$, $\de>0$, under which $u=(u^\a)_{\a=1}^d\in(\mathcal{C}_0^{\delta-1})^d$ has the covariance
$$
E[u^\a(x)u^\b(y)]=A^{\a\b}\delta(x-y).
$$
Note that $\mu_A$ is the distribution of $(\partial_x\si B)_{x\in \T}$, which is the limit in law of
that of  $\big(\partial_x(\si B*\eta^\e)\big)_{x\in \T}$ as $\e\downarrow 0$.
When $\si=I$, $\mu_A$ is called an $\R^d$-valued spatial white noise on $\T$.

\begin{theo} \label{3_thm:1.3}
Let $0<\delta<\delta'<\frac12$ and assume the trilinear condition \eqref{3_eq:1.2-s-hat}. Then there exists a subset $H\subset(\mathcal{C}_0^{\delta-1})^d$ such that $\mu_A(H)=1$, and if $\partial_xh(0)\in H$, the convergence to the limit process $h$ as above holds on whole $[0,\infty)$ (i.e., $h^\epsilon$ and $\tilde{h}^\epsilon$ exist in the space $C([0,\infty),(\mathcal{C}^\delta)^d)\cap C((0,\infty),(\mathcal{C}^{\delta'})^d)$ almost surely, and both of them converge to the same $h$ in the space $C([0,T],(\mathcal{C}^\delta)^d)\cap C([t,T],(\mathcal{C}^{\delta'})^d)$ for every $0<t<T$ in probability).  Moreover, the spatial derivative $u=\partial_xh$ of the limit process $h$ is a Markov process on $(\mathcal{C}_0^{\delta-1})^d$ which admits $\mu_A$ as an invariant measure.  
\end{theo}

\begin{rem}
Proposition 5.4 of Hairer and Mattingly \cite{HM16} (combined with Theorem \ref{3_thm:1.3}) shows that the limit process $h$ exists on $[0,\infty)$ almost surely for all initial values $h(0) \in(\mathcal{C}^\delta)^d$, since the measure $\mu_A$ has a dense support in $(\mathcal{C}_0^{\delta-1})^d$.
\end{rem}

Finally in this subsection, we note that the Cole-Hopf transform works for the coupled KPZ equation \eqref{3_eq:KPZ} in special cases. For example, Erta\c{s} and Kardar \cite{EK} considered the $\R^2$-valued coupled equations
\begin{equation}\label{3_EK:d=2}
\begin{aligned}
\partial_th^1&=\tfrac12\partial_x^2h^1+\tfrac12\{\lambda_1(\partial_xh^1)^2+\lambda_2(\partial_xh^2)^2\}+\si_1\xi^1,\\
\partial_th^2&=\tfrac12\partial_x^2h^2+\lambda_1\partial_xh^1\partial_xh^2+\si_2\xi^2
\end{aligned}
\end{equation}
as a linearizable case. In general, if we assume that there exists an invertible matrix $s=(s^\a_\b)_{1\le \a,\b\le d}$ ($s$ may be complex valued) such that
\begin{align}\label{3_Cole Hopf}
\Ga_{\b\ga}^\a=\sum_{\a'}(s^{-1})_{\a'}^\a s_\b^{\a'}s_\ga^{\a'},
\end{align}
then $\hat h^\alpha=s^\alpha_\beta h^\beta$ defined from the solution $h$ of \eqref{3_eq:KPZ} satisfies
\begin{align}  \label{3_eq:1.12}
\partial_t\hat h^\a&=\tfrac12\partial_x^2\hat h^\a+\tfrac12s^\a_{\a'}\Ga_{\b\ga}^{\a'}\partial_xh^\b\partial_xh^\ga+s^\a_\b\si_\ga^\b\xi^\ga\\
&=\tfrac12\partial_x^2\hat h^\a+\tfrac12(\partial_x\hat h^\a)^2+s^\a_\b\si_\ga^\b\xi^\ga.
\notag
\end{align}
In this way, the nonlinear term is decoupled. 
Hence the Cole-Hopf transform $Z^\a=\exp{\hat h^\a}$ linearizes \eqref{3_eq:KPZ}, so that the argument in \cite{GP} yields the global existence of $h$. In fact, the equation \eqref{3_EK:d=2} satisfies the condition \eqref{3_Cole Hopf} with $s=\bigl(\begin{smallmatrix}\lambda_1&(\lambda_1\lambda_2)^{1/2}\\\lambda_1&-(\lambda_1\lambda_2)^{1/2}\end{smallmatrix}\bigr)$. Meanwhile, even if $d=2$, the matrices $\Ga^1=\bigl(\begin{smallmatrix}2&1\\1&1\end{smallmatrix}\bigr)$ and $\Ga^2=\bigl(\begin{smallmatrix}1&1\\1&2\end{smallmatrix}\bigr)$ satisfy the trilinear condition
\eqref{3_eq:1.2-s-hat}, but not \eqref{3_Cole Hopf}.

As for the invariant measure, the tilt process $\partial_x \hat h^\a$ of each component $\hat h^\a$
of the transformed process has the distribution $\mu_\a$ of $\big(\sqrt{\sum_\ga (s^\a_\b\si_\ga^\b)^2}
\partial_x B\big)_{x\in\T}$ as its invariant measure, where $B$ is the $1$-dimensional periodic Brownian motion.
This is seen by applying the result of \cite{FQ} or Theorem \ref{3_thm:1.3} stated above for each component noting that 
$s^\a_\b\si_\ga^\b\xi^\ga$ in \eqref{3_eq:1.12} is the scalar-valued space-time white noise with covariance
$\sum_\ga (s^\a_\b\si_\ga^\b)^2$.  In particular, if \eqref{3_Cole Hopf} holds, we see that the tilt process
$u=\partial_x h$ of the solution $h=(h^\a)$ of \eqref{3_eq:KPZ} in the limit with a suitable renormalization
has an invariant measure, whose marginals under the transform $\hat h= sh$ is
given by $\mu_\a$ for each $\a$.  Indeed, with the help of Rellich type theorem,
one can easily show the tightness on the space
$(\mathcal{C}_0^{\delta-1})^d$ of the Ces$\grave{a}$ro mean $\mu_T=\frac1T\int_0^T \mu(t)dt$
over $[0,T]$ of the distributions $\mu(t)$ of $\partial_x \hat h(t)$ having an initial distribution
$\otimes_\a \mu_\a$, so that the limit distribution of $\mu_T$ as $T\to\infty$ is the invariant
measure.  However, the joint distribution of such invariant measure is unclear.

\begin{rem}
As we stated in Theorem \ref{3_thm:1.2} and will see in Lemma \ref{lemm:expression of G and F} below, the trilinear condition \eqref{3_eq:1.2-s-hat} is equivalent to the condition ``$F=G$ as matrices", which is also equivalent to that the logarithmic renormalization factors $B^{\e,\b\ga}$ and $\tilde B^{\e,\b\ga}$ behave as $O(1)$; indeed, $\tilde B^{\e,\b\ga}=0$ under this condition.
However, the necessary and sufficient condition for the logarithmic renormalization factors staying bounded in the KPZ approximating equations \eqref{3_eq:1-2} and \eqref{3_eq:1} is that the quantities $\Gamma_{\beta\gamma}^\alpha B^{\epsilon,\beta\gamma}$ or $\Gamma_{\beta\gamma}^\alpha \tilde{B}^{\epsilon,\beta\gamma}$, rather than $B^{\e,\b\ga}$ or $\tilde B^{\e,\b\ga}$ themselves, are bounded for every $\alpha$, respectively. From the expressions given in Lemma \ref{lemm:expression of G and F}, this is equivalent to that the identity
\begin{align}\label{true condition for no log}
\hat{\Gamma}_{\alpha_1\alpha_2}^\alpha\hat{\Gamma}_{\alpha_3\alpha_4}^{\alpha_1}\hat{\Gamma}_{\alpha_3\alpha_4}^{\alpha_2}
=\hat{\Gamma}_{\alpha_1\alpha_2}^\alpha\hat{\Gamma}_{\alpha_3\alpha_4}^{\alpha_1}\hat{\Gamma}_{\alpha_2\alpha_4}^{\alpha_3},
\end{align}
where the sums $\sum$ over $(\alpha_1,\alpha_2,\alpha_3,\alpha_4)$ are omitted, holds for every $\alpha$. Both \eqref{3_eq:1.2-s-hat} and \eqref{3_Cole Hopf} are sufficient conditions of \eqref{true condition for no log}, but neither of them are necessary conditions.
In particular, the logarithmic renormalization factor does not appear in the equation \eqref{3_eq:1.12}.
\end{rem}

\subsection{Notations and organization of the paper}

The Fourier transform on $\R$ is denoted by $\fa= \mathcal{F}\eta\in \mathcal{S}(\R)$, i.e. $\fa(\th) = \int_\R e^{-2\pi ix\th}\eta(x)dx$, $\th \in \R$. When $\eta$ is even and satisfies $\int_\R\eta(x)=1$, then $\fa$ is real-valued and satisfies $\fa(0)=1$ and $\fa(\th) = \fa(-\th)$. The convolution operators $*\eta^\e$ and $*\eta_2^\e$ are represented by the Fourier multipliers $\fa(\e D)$ and $\fa^2(\e D)$, where $\fa(D)u:=\mathcal{F}^{-1}(\fa\mathcal{F}u)$.

We also consider the Fourier transform on $\T$ and use the same notation $\mathcal{F}$ and $\mathcal{F}^{-1}$:
\begin{align*}
&\mathcal{F} u(k) =\hat u(k) = \int_\T e^{-2\pi i kx}u(x)dx, \quad k\in \Z, \\
&\mathcal{F}^{-1}v(x) = \sum_k e^{2\pi i kx} v(k),\quad x\in\R.
\end{align*}
Then, the heat kernel associated with $\partial_t-\frac12\partial_x^2$ is given by
$$
p(t,x) = \sum_k e^{2\pi i kx} e^{-2\pi^2 k^2t}\bigg(= \mathcal{F}^{-1} \big( e^{-2\pi^2 k^2t}\big)(x) \bigg),\quad t>0, \; x\in \T.
$$
The noises $\xi^\b(t,x)$ are transformed into complex-valued white noises
$\xi^{\b,k}(s)= (\mathcal{F} \xi^\b(s,\cdot))(k)$ such that $\overline{\xi^{\b,k}(s)}=\xi^{\b,-k}(s)$ and
\begin{equation}  \label{3_eq:xicov}
E[\xi^{\b,k_1}(t)\xi^{\ga,k_2}(s)]
= \de^{\b\ga} \de(t-s) 1_{\{k_1+k_2=0\}}.
\end{equation}
In fact, the left hand side of \eqref{3_eq:xicov} is given by
$$
E\left[ \int_\T e^{-2\pi i k_1 x} \xi^\b(t,x) dx
\int_\T e^{-2\pi i k_2 y} \xi^\ga(s,y) dy \right]
= \de^{\b\ga} \de(t-s) \int_\T  e^{-2\pi i k_1 x}  e^{-2\pi i k_2 x} dx.
$$
The smeared noise is defined by $\xi*\eta^\e = \fa(\e D) \xi$, where
$\fa(D)u=\mathcal{F}^{-1}(\fa \mathcal{F} u)$ as we mentioned above.

This paper is organized as follows. In Section \ref{3_section 2}, following \cite{Ho} we formulate a fixed point problem associated with \eqref{3_eq:KPZ} and solve it by constructing a deterministic solution map from the initial value and deterministic driving terms. In Section \ref{3_section 3}, we prove the probabilistic part of Theorem \ref{3_thm:1.1}, i.e., we give controls of stochastic drivers and calculations of renormalization factors. In Section \ref{3_section 4}, we prove Theorem \ref{3_thm:1.2} under the trilinear condition \eqref{3_eq:1.2-s-hat}. In Section \ref{3_section 5}, we repeat the same arguments as in Section \ref{3_section 2} for the stochastic Burgers equation:
$$
\partial_t u^\a = \tfrac12 \partial_x^2 u^\a + \tfrac12 \Ga_{\b\ga}^\a
\partial_x (u^\b u^\ga) + \si^\a_\b \partial_x\xi^\b,
$$
and construct a well-defined solution map. We show the invariance of $\mu_A$ under \eqref{3_eq:KPZ} at the Burgers level and prove Theorem \ref{3_thm:1.3}. At last, we touch the global well-posedness of the approximating equations \eqref{3_eq:1-2} and \eqref{3_eq:1} at the Burgers level.


\section{Formal expansion and solving the coupled KPZ equation}\label{3_section 2}

\subsection{Preliminary consideration due to formal expansion}

In the coupled KPZ equation \eqref{3_eq:KPZ}, we think of the noise as the leading term and the nonlinear term as its perturbation.
Although we  eventually take $a=1$, we put $a>0$ in front of the nonlinear term:
\begin{equation}  \label{3_eq:2.1}
\mathcal{L} h^\a = \frac{a}2 \Ga_{\b\ga}^\a
\partial_x h^\b \partial_x h^\ga + \si^\a_\b \xi^\b,
\end{equation}
where $\mathcal{L} = \partial_t- \tfrac12 \partial_x^2$. Then, at least formally, one can expand the solution $h$ in $a$:
\begin{equation}  \label{3_eq:2.2}
h^\a = \sum_{k=0}^\infty a^k h_k^\a.
\end{equation}
Indeed, by inserting \eqref{3_eq:2.2} to \eqref{3_eq:2.1}, we have that
$$
\sum_{k=0}^\infty a^k \mathcal{L} h_k^\a= \si^\a_\b \xi^\b + \frac{a}2 \sum_{k_1,k_2=0}^\infty a^{k_1+k_2} \Ga_{\b\ga}^\a \partial_x h^\b_{k_1} \partial_x h^\ga_{k_2}.
$$
Thus, comparing the terms of order $a^0, a^1, a^2, a^3$ in both sides and noting the condition \eqref{3_eq:1.2}, we obtain the following identities:
\begin{equation}  \label{3_eq:2.3}
\begin{aligned}
\mathcal{L} h_0^\a &= \si^\a_\b \xi^\b, \\
\mathcal{L} h_1^\a &= \tfrac12 \Ga_{\b\ga}^\a 
\partial_x h^\b_0 \partial_x h^\ga_0, \\
\mathcal{L} h_2^\a &= \Ga_{\b\ga}^\a 
\partial_x h^\b_1 \partial_x h^\ga_0, \\
\mathcal{L} h_3^\a &= 
\tfrac12 \Ga_{\b\ga}^\a 
\partial_x h^\b_1 \partial_x h^\ga_1
+ \Ga_{\b\ga}^\a 
\partial_x h^\b_2 \partial_x h^\ga_0.
\end{aligned}
\end{equation}
The first equation determines $h_0^\a$, which is actually the Ornstein-Uhlenbeck process, and $h_0^\a\in \mathcal{C}^{1/2-}
:= \bigcap_{\de>0} \mathcal{C}^{1/2-\de}$ in $x$. Therefore, the product $\partial_x h^\b_0 \partial_x h^\ga_0$ in the second equation is not definable in a usual sense.  When $\xi^\b$ is replaced by the smeared noise $\xi^{\e,\b}:=\xi^\b*\eta^\e$, this product makes sense, since $h_0^\a \in C^\infty$ for such case. However, as we will see later in \eqref{3_eq:2.4}, for $h_1^\a$ to converge, we need to introduce a renormalization. At this moment, we just assume $h_1^\a \in \mathcal{C}^{1-}$ (note $-\tfrac12-\tfrac12+2=1)$ and then $h_2^\a \in \mathcal{C}^{3/2-}$ (note $-\tfrac12+0+2=\frac32)$ are defined in some sense. We denote $h_0^\a, h_1^\a, h_2^\a$ with stationary initial values by $H_\I^\a, H_\Y^\a, H_\W^\a$, respectively, see Section \ref{3_section 3} for details.

After defining $H_\I^\a, H_\Y^\a, H_\W^\a$, the KPZ equation (with $a=1$) for $h^\a =
H_\I^\a+ H_\Y^\a+ H_\W^\a+ h_{\ge 3}^\a$ can be rewritten into an equation for the remainder term $h_{\ge 3}$:
\begin{align}  \label{3_eq:h_3}
\mathcal{L} h_{\ge 3}^\a = 
\Phi^\a + \mathcal{L} h_3^\a, 
\end{align}
where $\Phi^\a=\Phi^\a(H_\I, H_\Y, H_\W, h_{\ge 3})$ is given by
\begin{align*}
\Phi^\a& = \Ga_{\b\ga}^\a \partial_x h^\b_{\ge 3} \partial_x H^\ga_\I
+ \Ga_{\b\ga}^\a (\partial_x H^\b_\W +\partial_x h^\b_{\ge 3})
\partial_x H^\ga_\Y \\
&\quad+ \tfrac12 \Ga_{\b\ga}^\a (\partial_x H^\b_\W +\partial_x h^\b_{\ge 3})
(\partial_x H^\ga_\W +\partial_x h^\ga_{\ge 3}).
\end{align*}
This is easily obtained from \eqref{3_eq:KPZ} and \eqref{3_eq:2.3} by computing $\mathcal{L}h^\a-\mathcal{L}H_\I^\a-\mathcal{L}H_\Y^\a-\mathcal{L}H_\W^\a$ and expanding $\mathcal{L}h^\a-\mathcal{L}H_\I^\a = \tfrac12 \Ga_{\b\ga}^\a\partial_x (H_\I^\b+H_\Y^\b+H_\W^\b+h_{\ge 3}^\b) \partial_x (H_\I^\ga+H_\Y^\ga+H_\W^\ga+ h_{\ge 3}^\ga)$. Note that $h_3^\a=h_3^\a(H_\I,H_\Y,H_\W)$ in \eqref{3_eq:h_3} is defined through the last identity in \eqref{3_eq:2.3}.

We now recall that Section 2 of \cite{Ho} briefly summarizes definitions and known results on Besov spaces, 
Bony's paraproducts $u\pl v$, $u\rs v$ of $u$ and $v$,
mollifier estimates, Schauder estimates, commutator estimates and others;
see \cite{GIP,GP} for details. 
To define $h_{\ge 3}^\a$, we need to introduce four more objects as driving terms:
\begin{align*}
H_\Bo^{\b\ga} &= \tfrac12 \partial_x H^\b_\Y \partial_x H^\ga_\Y,
&H_\D^{\b\ga} &= \partial_x H^\b_\W \rs \partial_x H^\ga_\I,\\
H_\K^\a &= \lq\lq\text{stationary solution of } \mathcal{L} H_\K^\a= \partial_x H^\a_\I",
&H_\Co^{\b\ga} &= \partial_x H_\K^\b \rs \partial_x H^\ga_\I.
\end{align*}
Indeed, to solve the equation \eqref{3_eq:h_3}, we divide $h_{\ge 3}^\a$ into the sum of two parts $f^\a$ and $g^\a$: $h_{\ge 3}^\a=f^\a+g^\a$, which solve
\begin{equation}\label{3_eq:fg}
\begin{aligned}
\mathcal{L} f^\a &= \Ga_{\b\ga}^\a 
(\partial_x H_\W^\b+\partial_x f^\b+\partial_x g^\b)\pl
\partial_x H^\ga_\I, \\
\mathcal{L} g^\a & = \Ga_{\b\ga}^\a 
(\partial_x H_\W^\b+\partial_x f^\b+\partial_x g^\b)(\rs+\pr)
\partial_x H^\ga_\I+\text{other terms},
\end{aligned}
\end{equation}
respectively.  Here, the implicitly written \lq\lq other terms" contain nonlinear operators of sufficiently regular
functions, so that they are well-defined if we can define $H_\Bo^{\b\ga},H_\D^{\b\ga}\in\mathcal{C}^{0-}$. To define the term $\partial_xf^\b\rs\partial_xH^\ga_\I$ in the right hand side of $\mathcal{L}g^\a$, we first introduce $H_\K^\a$ as a solution of $\mathcal{L} H_\K^\a= \partial_x H^\a_\I$.  Then by definition, $f^\a$ has the form
$$
f^\a = P_\cdot f^\a(0) + \Ga_{\b\ga}^\a(\partial_xH_\W^\b+\partial_xh_{\ge3}^\b) \pl (H_\K^\ga- P_\cdot H_\K^\ga(0))
+ C_1^\alpha(H_\K,h_{\ge3},H_\I),
$$
with a term $C_1^\alpha(H_\K,h_{\ge3},H_\I)$ sufficiently regular in the sense that the resonant $\partial_xC_1^\beta\rs\partial_xH_\I^\gamma$ is well-defined, see Lemma 3.1 of \cite{Ho}. Here $P\equiv P_t$ is the heat semigroup defined by $P_tu=\int_\T p(t,\cdot-y)u(y)dy$. By the commutator estimate, e.g., see Lemma 2.4 of \cite{GIP} or 
Proposition 2.12 of \cite{Ho}, the term $\partial_x f^\b\rs \partial_x H^\ga_\I$ is defined if $H_\Co^{\b\ga} = \partial_x H_\K^\b\rs\partial_x H^\ga_\I\in\mathcal{C}^{0-}$ is given
a priori.

\subsection{Drivers}

Fix $\k\in (\frac13,\tfrac12)$. The driver of the coupled KPZ equation is the element ${\mathbb H}$ of the form
\begin{align*}
&{\mathbb H} :=(H_\I,H_\Y,H_\W,H_\Bo,
H_\D,H_\K,H_\Co)\\
&\in C([0,T],(\mathcal{C}^\k)^d)
\times C([0,T],(\mathcal{C}^{2\k})^d)
\times \{ C([0,T],(\mathcal{C}^{\k+1})^d) \cap C^{1/4}([0,T],(\mathcal{C}^{\k+1/2})^d)\} \\
&\quad\times C([0,T],(\mathcal{C}^{2\k-1})^{d^2}) \times C([0,T],(\mathcal{C}^{2\k-1})^{d^2}) 
\times C([0,T],(\mathcal{C}^{\k+1})^d) \times C([0,T],(\mathcal{C}^{2\k-1})^{d^2}),
\end{align*}
which satisfies $\mathcal{L}H_\K=\partial_xH_\I$. Note that, for $H_\Bo$ and others, the Besov space is $\R^d\otimes\R^d$-valued. We denote by $\mathcal{H}_{\text{KPZ}}^\k$ the class of all drivers. We write $\$\mathbb{H}\$_T$ for the product norm of $\mathbb{H}$ on the above space. Due to the preliminary and heuristic consideration, these terms should be defined a priori in some sense.  

\subsection{Deterministic result}

Fix $\la\in(\tfrac13,\k)$ and $\mu\in(-\la,\la]$. For a $\mathcal{D}'(\T,\R^d)$-valued functions $f=(f^\a)_{\a=1}^d$
and $g=(g^\a)_{\a=1}^d$ on $[0,T]$, we write $(f,g)\in\mathcal{D}_{\text{KPZ}}^{\la,\mu}([0,T])$ if
\begin{align*}
&\|(f,g)\|_{\mathcal{D}_{\text{KPZ}}^{\la,\mu}([0,T])}:=\\
&\sup_{t\in[0,T]}t^{\frac{\la-\mu}2}\|f(t)\|_{(\mathcal{C}^{\la+1})^d}
+\sup_{t\in[0,1]}\|f(t)\|_{(\mathcal{C}^{\mu+1})^d}
+\sup_{s<t\in[0,T]}s^{\frac{\la-\mu}2}\frac{\|f(t)-f(s)\|_{(\mathcal{C}^{\la+1/2})^d}}{|t-s|^{1/4}}\\
&+\sup_{t\in[0,T]}t^{\la-\mu}\|g(t)\|_{(\mathcal{C}^{2\la+1})^d}
+\sup_{t\in[0,1]}\|g(t)\|_{(\mathcal{C}^{2\mu+1})^d}
+\sup_{s<t\in[0,T]}s^{\la-\mu}\frac{\|g(t)-g(s)\|_{(\mathcal{C}^{2\la+1/2})^d}}{|t-s|^{1/4}}
\end{align*}
is finite.

The following theorem is due to the paracontrolled calculus and fixed point theorem. For the detailed proof, see Section 3 of \cite{Ho}.

\begin{theo}[Theorem 3.3 and Lemma 3.5 of \cite{Ho}]  \label{3_thm:3}
{\rm (1)} Let $T>0$ and ${\mathbb H}\in \mathcal{H}_{\text{\rm KPZ}}^\k$ be given. Then, for every initial value $(f(0),g(0))\in(\mathcal{C}^{\mu+1})^d\times(\mathcal{C}^{2\mu+1})^d$ the system \eqref{3_eq:fg} admits a unique solution in $\mathcal{D}_{\text{\rm KPZ}}^{\la,\mu}([0,T_*])$ up to the time
$$
T_*=C(1+\|f(0)\|_{(\mathcal{C}^{\mu+1})^d}+\|g(0)\|_{(\mathcal{C}^{2\mu+1})^d}+\$\mathbb{H}\$_{T}^3)^{-\frac2{\k-\la}}\wedge T,
$$
where $C$ is a universal constant depending only on $\k,\la,\mu$ and $T$. The solution satisfies
$$
\|(f,g)\|_{\mathcal{D}_{\text{\rm KPZ}}^{\la,\mu}([0,T_*])}\le C'(1+\|f(0)\|_{(\mathcal{C}^{\mu+1})^d}+\|g(0)\|_{(\mathcal{C}^{2\mu+1})^d}+\$\mathbb{H}\$_{T}^3),
$$
with a universal constant $C'$.  \\
{\rm (2)} Let $T_{\text{\rm sur}}\le T$ be the maximal time such that the existence and uniqueness of the solution hold on $[0,T_{\text{\rm sur}})$. The map $(f(0),g(0),\mathbb{H})\mapsto T_{\text{\rm sur}}$ is lower semi-continuous. If $T_{\text{\rm sur}}<T$, then
$$
\lim_{t\uparrow T_{\text{\rm sur}}}\|h\|_{C([0,t],(\mathcal{C}^{\k\wedge(\mu+1)\wedge(2\mu+1)})^d)}=\infty,
$$
where $h=H_\I+H_\Y+H_\W+f+g$. The map $S_{\rm KPZ}$ from $(f(0),g(0),{\mathbb H})\in(\mathcal{C}^{\mu+1})^d\times(\mathcal{C}^{2\mu+1})^d\times\mathcal{H}_{\text{\rm KPZ}}^\k$ to
the maximal solution $h\in C([0,T_{\text{\rm sur}}),(\mathcal{C}^{\k\wedge(\mu+1)\wedge(2\mu+1)})^d)$ is continuous.
\end{theo}

We denote by $h=H_\I+H_\Y+H_\W+f+g\equiv S_{\rm KPZ}(f(0),g(0),\mathbb{H})$ the maximal solution.

\subsection{Renormalization}

By replacing $\xi^\b$ by $\xi^{\e,\b}=\xi^\b*\eta^\e$ in \eqref{3_eq:KPZ} and introducing the renormalization factors $- c^\e A^{\b\ga}$, $C^{\e,\b\ga}$ and $D^{\e,\b\ga}$, we have the following identities for the formal expansion $h^{\e,\a} = \sum_{k=0}^\infty a^k h_k^{\e,\a}$ of the solution of the approximating equation \eqref{3_eq:1-2}:
\begin{equation*} 
\begin{aligned}
\mathcal{L} h_0^{\e,\a} &= \si^\a_\b \xi^{\e,\b}, \\
\mathcal{L} h_1^{\e,\a} &= \tfrac12 \Ga_{\b\ga}^\a 
(\partial_x h^{\e,\b}_0 \partial_x h^{\e,\ga}_0 - c^\e A^{\b\ga}), \\
\mathcal{L} h_2^{\e,\a} &= \Ga_{\b\ga}^\a 
\partial_x h^{\e,\b}_1 \partial_x h^{\e,\ga}_0, \\
\mathcal{L} h_3^{\e,\a} &= 
\tfrac12 \Ga_{\b\ga}^\a (\partial_x h^{\e,\b}_1 \partial_x h^{\e,\ga}_1
-C^{\e,\b\ga}) + \Ga_{\b\ga}^\a 
(\partial_x h^{\e,\b}_2 \partial_x h^{\e,\ga}_0 - D^{\e,\b\ga}).
\end{aligned}
\end{equation*}
Then we obtain the renormalized driver $\mathbb{H}^\e$ corresponding to $\xi^\e$,
which is defined by the solutions of
\begin{equation} \label{3_eq:2.4}
\begin{aligned}
\mathcal{L} H_\I^{\e,\a} &= \si^\a_\b \partial_x\xi^{\e,\b}, \\
\mathcal{L} H_\Y^{\e,\a} &= \tfrac12 \Ga_{\b\ga}^\a 
(\partial_xH_\I^{\e,\b} \partial_xH_\I^{\e,\ga}-c^\e A^{\b\ga}), \\
\mathcal{L} H_\W^{\e,\a} &= \Ga_{\b\ga}^\a 
\partial_xH_\Y^{\e,\b} \partial_xH_\I^{\e,\ga},\\
\mathcal{L} H_\K^{\e,\a} &=\partial_xH_\I^{\e,\a}
\end{aligned}
\end{equation}
with stationary initial values, and products
\begin{equation*}
\begin{aligned}
H_\Bo^{\e,\b\ga}&=\tfrac12 (\partial_xH_\Y^{\e,\b}\partial_xH_\Y^{\e,\ga}-C^{\e,\b\ga}),\\
H_\D^{\e,\b\ga}&=\partial_xH_\W^{\e,\b}\rs \partial_xH_\I^{\e,\ga}-D^{\e,\b\ga},\\
H_\Co^{\e,\b\ga}&=\partial_xH_\K^{\e,\b}\rs \partial_xH_\I^{\e,\ga}.
\end{aligned}
\end{equation*}
From this, we see that $h^\e := S_{\text{\rm KPZ}}(f(0),g(0),{\mathbb H}^\e)$ solves \eqref{3_eq:1-2} with
$$
B^{\e,\b\ga} = C^{\e,\b\ga}+2D^{\e,\b\ga}.
$$
Theorem \ref{3_thm:3} combined with the convergence of drivers ${\mathbb H}^\e$ to ${\mathbb H}$
shown in Theorem \ref{3_conv of driv} below proves Theorem \ref{3_thm:1.1}-(1).

We do similar arguments for the equation with $*\eta_2^\e=\fa^2(\e D)$ for the nonlinear term:
\begin{align}\label{3_eq:fa in sq}
\partial_t\tilde h^\a=\tfrac12\partial_x^2\tilde h^\a+\tfrac12\Ga_{\b\ga}^\a\fa^2(\e D)(\partial_x\tilde h^\b\partial_x\tilde h^\ga)+\si_\b^\a\xi^\b.
\end{align}
Then for the formal expansion $\tilde h^{\a} = \sum_{k=0}^\infty a^k \tilde h_k^{\a}$, we have
\begin{equation*}
\begin{aligned}
& \mathcal{L} \tilde h_0^\a = \si^\a_\b \xi^\b, \\
& \mathcal{L} \tilde h_1^\a = \tfrac12 \Ga_{\b\ga}^\a 
\fa^2(\e D)(\partial_x \tilde h^\b_0 \partial_x \tilde h^\ga_0), \\
& \mathcal{L} \tilde h_2^\a = \Ga_{\b\ga}^\a 
\fa^2(\e D)(\partial_x \tilde h^\b_1 \partial_x \tilde h^\ga_0), \\
& \mathcal{L} \tilde h_3^\a = 
\tfrac12 \Ga_{\b\ga}^\a \fa^2(\e D)(\partial_x \tilde h^\b_1 \partial_x \tilde h^\ga_1
)
+\Ga_{\b\ga}^\a \fa^2(\e D)
(\partial_x \tilde h^\b_2 \partial_x \tilde h^\ga_0).
\end{aligned}
\end{equation*}
We can construct a solution map $h=S_{\text{\rm KPZ}}^\e(f(0),g(0),\mathbb{H})$ corresponding to the mollified equation \eqref{3_eq:fa in sq}, though the driver $\mathbb{H}$ satisfies $\mathcal{L}H_\K=\partial_x\fa^2(\e D)H_\I$. See Section 4 of \cite{Ho} for the scalar-valued case. Furthermore, we have the following convergence result.

\begin{theo}[Theorem 4.4 of \cite{Ho}]  \label{3_thm:5}
If $(f^\e(0),g^\e(0)) \to (f(0),g(0))$ in $(\mathcal{C}^{\mu+1})^d\times(\mathcal{C}^{2\mu+1})^d$ and 
${\mathbb H}^\e \to {\mathbb H}$ in $\mathcal{H}_{\rm KPZ}^\k$, 
then we have the convergence $S_{\text{\rm KPZ}}^\e(f^\e(0),g^\e(0),{\mathbb H}^\e)
\to S_{\text{\rm KPZ}}(f(0),g(0),{\mathbb H})$ in $C([0,T_{\text{\rm sur}}),(\mathcal{C}^{\k\wedge(\mu+1)\wedge(2\mu+1)})^d)$.
\end{theo}

By replacing $\xi^\b$ by $\xi^{\e,\b}$ in \eqref{3_eq:fa in sq}
and introducing the renormalization factors $- c^\e A^{\b\ga}, \tilde C^{\e,\b\ga}, \tilde D^{\e,\b\ga}$, we again obtain the renormalized driver $\tilde{\mathbb{H}}^\e$ corresponding to the approximating equation \eqref{3_eq:1}, which is defined by the solutions of
\begin{equation}\label{3_eq:2.5}
\begin{aligned}
\mathcal{L} \tilde H_\I^{\e,\a} &= \si^\a_\b \partial_x\xi^{\e,\b}, \\
\mathcal{L} \tilde H_\Y^{\e,\a} &= \tfrac12 \Ga_{\b\ga}^\a 
\fa^2(\e D)(\partial_x\tilde H_\I^{\e,\b} \partial_x\tilde H_\I^{\e,\ga}-c^\e A^{\b\ga}), \\
\mathcal{L} \tilde H_\W^{\e,\a} &= \Ga_{\b\ga}^\a 
\fa^2(\e D)(\partial_x\tilde H_\Y^{\e,\b} \partial_x\tilde H_\I^{\e,\ga}),\\
\mathcal{L} \tilde H_\K^{\e,\a} &=\partial_x\fa^2(\e D)\tilde H_\I^{\e,\a}
\end{aligned}
\end{equation}
with stationary initial values, and products
\begin{equation*}
\begin{aligned}
\tilde H_\Bo^{\e,\b\ga}&=\tfrac12 (\partial_x\tilde H_\Y^{\e,\b}\partial_x\tilde H_\Y^{\e,\ga}-\tilde C^{\e,\b\ga}),\\
\tilde H_\D^{\e,\b\ga}&=\partial_x\tilde H_\W^{\e,\b}\rs \partial_x\tilde H_\I^{\e,\ga}-\tilde D^{\e,\b\ga},\\
\tilde H_\Co^{\e,\b\ga}&=\partial_x\tilde H_\K^{\e,\b}\rs \partial_x\tilde H_\I^{\e,\ga}.
\end{aligned}
\end{equation*}
From this, we see that 
$\tilde h^\e  := S_{\text{\rm KPZ}}^\e(f(0),g(0),\tilde{\mathbb H}^\e)$
solves \eqref{3_eq:1} with
$$
\tilde B^{\e,\b\ga}= \tilde C^{\e,\b\ga}+ 2\tilde D^{\e,\b\ga}.
$$
Theorem \ref{3_thm:5} combined with the convergence of drivers 
$\tilde{\mathbb H}^\e$ to ${\mathbb H}$ shown in Theorem \ref{3_conv of driv}
proves Theorem \ref{3_thm:1.1}-(2).


\section{Computation of renormalization factors}\label{3_section 3}

\subsection{Product formula}

We first prepare the product formula to compute the Wiener chaos expansions
of the products of two multiple Wiener-It\^o integrals.

Let $\{\xi^{\b,k}(s)\}_{\b\in\{1,\ldots,d\},k\in\Z}$ be complex-valued Gaussian white noises on $\R$ which satisfy $\overline{\xi^{\b,k}(s)}=\xi^{\b,-k}(s)$ and have the covariance structure \eqref{3_eq:xicov}, which are realized on a probability space $(\Omega,\mathcal{F},P)$, where $\mathcal{F}$ is a $\sigma$-field generated by $\{\xi^{\beta,k}(\mathbf{1}_{(s,t]})\}_{\beta,k,s<t}$.
The Hilbert space $\mathcal{H}=L^2(\Omega, \mathcal{F},P)$ can be decomposed into the direct sum $\oplus_\bn \mathcal{H}_{\bn}$, where $\bn = (n_{\b,k})\in \Z_{\ge0}^{\{1,\ldots,d\}\times\Z}$ satisfies $|\bn| :=\sum n_{\b,k} <\infty$ and $\Z_{\ge0}= \N \cup\{0\}$.
For $f\Big( \big((s_i^{\b,k})_{i=1}^{n_{\b,k}}\big)_{\b,k}\Big) \in\mathcal{H}_{\bn}:= L^2(\R^{|\bn|},ds_\bn)$ with $ds_\bn = \prod_{\b,k}\prod_{i=1}^{n_{\b,k}}ds_i^{\b,k}$, we define multiple Wiener-It\^o integrals
$$
I_{\bn}(f) = \int f\Big( \big((s_i^{\b,k})_{i=1}^{n_{\b,k}}\big)_{\b,k}\Big)
\prod_{\b,k}\prod_{i=1}^{n_{\b,k}} \xi^{\b,k}(ds_i^{\b,k}).
$$
Note that we don't divide by $\bn! = \prod_{\b,k}n_{\b,k}!$ in the definition of $I_{\bn}(f)$ compared with \cite{Ma}.

For given $\bn=(n_{\b,k})$ and $\bm=(m_{\b,k})$, a diagram $\la\subset\coprod_{\b,k}\{1,\ldots,n_{\b,k}\}\times\coprod_{\b,k}\{1,\ldots,m_{\b,k}\}$ consists of disjoint pairs $(i_{\b,k},j_{\b,-k})$ connecting $i_{\b,k}\in\{1,\ldots,n_{\b,k}\}$ and $j_{\b,-k}\in\{1,\ldots,m_{\b,-k}\}$. Note that $\b$ are common and $k$ have opposite signs in these pairs. The set of all possible diagrams $\{\la\}$ is denoted by $\Ga(\bn,\bm)$. For $\la\in \Ga(\bn,\bm)$, we denote $\bar\la = (\bar\la_{\b,k})$, where $\bar\la_{\b,k}= \sharp\{i_{\b,k} \in \la\}+\sharp\{j_{\b,k} \in \la\}$ for each $\b, k$. Then $\bn+\bm -\bar\la$ is defined componentwisely by $(\bn+\bm -\bar\la)_{\b,k} = n_{\b,k}+m_{\b,k}-\bar\la_{\b,k}$. For $f_1\in\mathcal{H}_\bn$ and $f_2\in\mathcal{H}_\bm$, we define $f_\la\in \mathcal{H}_{\bn+\bm -\bar\la}$ by
$$
f_\la(s_\bn\sqcup s_\bm\setminus s_\la) = \int f_1(s_\bn) f_2(\check{s}_\bm)
\prod_{i_{\b,k}\in \la} ds_{i_{\b,k}}^{\b,k },
$$
where $s_\bn\sqcup s_\bm\setminus s_\la=\big((s_{i_{\b,k}}^{\b,k})_{i_{\b,k}\notin\la}\sqcup(s_{j_{\b,k}}^{\b,k})_{j_{\b,k}\notin\la}\big)_{\b,k}$, and $\check{s}_\bm$ is defined by $s_\bm$ with $s_{j_{\b, -k}}$ replaced by
$s_{i_{\b,k}}$ when the pair $(i_{\b,k}, j_{\b, -k})$ appears in $\la$.
Then we have the following product formula; see Theorem 5.3 of \cite{Ma} with $m=2$ shown in a slightly different setting from ours.

\begin{prop} \label{3_prop:prod}
For $f_1 \in \mathcal{H}_\bn, f_2 \in \mathcal{H}_\bm$, we have
$$
I_{\bn}(f_1) I_{\bm}(f_2) = \sum_{\la\in \Ga(\bn,\bm)} I_{\bn+\bm -\bar\la}(f_\la).
$$
\end{prop}

\subsection{Definition of driving processes}\label{3_section3.2}

Here we precisely define the components of $\mathbb{H}^\e$ and $\tilde{\mathbb H}^\e$. Now we write $\mathcal{H}_n=\oplus_{|\bn|=n}I_\bn(\mathcal{H}_\bn)$.

We define $I(\zeta)\equiv I(\zeta)(t,x)$ for a noise $\zeta=\{\zeta(s,y);s\in\R,y\in\T\}$ by
$$
I(\zeta)(t,x)  
= \int_{-\infty}^t \int_\T p(t-s,x,y) \Pi_0^\bot\zeta(s,y)dsdy
+\int_0^t \Pi_0\zeta(s)ds.
$$
Here $\Pi_0$ is the orthogonal projection of $L^2(\T)$ onto the set of constant functions, 
and $\Pi_0^\bot=1-\Pi_0$. $I(\zeta)$ is a solution of $\mathcal{L}I(\zeta)=\zeta$ with stationary initial value for non-zero Fourier modes but zero initial value for the zero mode. Note that $\partial_x I(\zeta)$ is a stationary process since
$$
\partial_x I(\zeta)(t,x)
=\int_{-\infty}^t \int_\T \partial_xp(t-s,x,y) \zeta(s,y)dsdy.
$$

By taking the smeared noise $\xi^{\e,\b}$, which is extended for $t\in \R$, set
\begin{align}\label{3_eq:def of HI}
& H_\I^{\e,\a}=\tilde H_\I^{\e,\a} := \si_\b^\a I(\xi^{\e,\b}) \in {\mathcal H}_1.
\end{align}
Note that this solves the first equations of \eqref{3_eq:2.4} and \eqref{3_eq:2.5}. This converges to the process
$H_\I^\a=\si_\b^\a I(\xi^\b)$ as $\e\downarrow 0$.

We define
\begin{equation}\label{3_eq:def of HY}
\begin{aligned}
&H_\Y^{\e,\a} :=
\tfrac12\Ga_{\b\ga}^\a I(\partial_xH_\I^{\e,\b}\partial_xH_\I^{\e,\ga}-c^\e A^{\b\ga}) \in {\mathcal H}_2,\\
&\tilde H_\Y^{\e,\a} :=
\tfrac12\Ga_{\b\ga}^\a \fa^2(\e D)I(\partial_xH_\I^{\e,\b}\partial_xH_\I^{\e,\ga}-c^\e A^{\b\ga}) \in {\mathcal H}_2.
\end{aligned}
\end{equation}
These solve the second equations of \eqref{3_eq:2.4} and \eqref{3_eq:2.5}, respectively.
Note that $\partial_xI(\xi^{\e,\b})$ is stationary in $t$, so that the $\mathcal{H}_0$-components of
$H_\Y^{\e,\a}$ and $\tilde H_\Y^{\e,\a}$ are compensated by $c^\e A^{\b\ga}$.

We define
\begin{equation}\label{3_eq:def of HW}
\begin{aligned}
&H_\W^{\e,\a}:=
\Ga_{\b\ga}^\a I(\partial_xH_\Y^{\e,\b}\partial_xH_\I^{\e,\ga})
\in {\mathcal H}_3\oplus \mathcal{H}_1,\\
&\tilde H_\W^{\e,\a}:=
\Ga_{\b\ga}^\a \fa^2(\e D)I(\partial_x\tilde H_\Y^{\e,\b}\partial_xH_\I^{\e,\ga})
\in {\mathcal H}_3\oplus \mathcal{H}_1.
\end{aligned}
\end{equation}
These solve the third equations of \eqref{3_eq:2.4} and \eqref{3_eq:2.5}, respectively. We can also define
\begin{equation*}
\begin{aligned}
&H_\Bo^{\e,\b\ga}=\tfrac12(\partial_xH_\Y^{\e,\b}\partial_xH_\Y^{\e,\ga}-C^{\e,\b\ga})
\in\mathcal{H}_4\oplus\mathcal{H}_2,\\
&\tilde H_\Bo^{\e,\b\ga}=\tfrac12(\partial_x\tilde H_\Y^{\e,\b}\partial_x\tilde H_\Y^{\e,\ga}-\tilde C^{\e,\b\ga})
\in\mathcal{H}_4\oplus\mathcal{H}_2,\\
&H_\D^{\e,\b\ga}=\partial_xH_\W^{\e,\b}\rs\partial_xH_\I^{\e,\ga}-D^{\e,\b\ga}
\in\mathcal{H}_4\oplus\mathcal{H}_2,\\
&\tilde H_\D^{\e,\b\ga}=\partial_x\tilde H_\W^{\e,\b}\rs\partial_xH_\I^{\e,\ga}-\tilde D^{\e,\b\ga}
\in\mathcal{H}_4\oplus\mathcal{H}_2,
\end{aligned}
\end{equation*}
by subtracting the corresponding $\mathcal{H}_0$-components.

We further define
\begin{equation*}
\begin{aligned}
&H_\K^{\e,\a} := I(\partial_xH_\I^{\e,\a})
\in \mathcal{H}_1,\\
&\tilde H_\K^{\e,\a} := \fa^2(\e D)I(\partial_xH_\I^{\e,\a})
\in \mathcal{H}_1,
\end{aligned}
\end{equation*}
and
\begin{equation*}
\begin{aligned}
& H_\Co^{\e,\b\ga} := 
\partial_x H_\K^{\e,\b} \rs \partial_x H^{\e,\ga}_\I
\in {\mathcal H}_2,\\
& \tilde H_\Co^{\e,\b\ga} :=
\partial_x \tilde H_\K^{\e,\b} \rs \partial_x H^{\e,\ga}_\I
\in {\mathcal H}_2.
\end{aligned}
\end{equation*}
Note that the $\mathcal{H}_0$-components vanish because the function $\fa=\mathcal{F}\eta$ is even.

The following result is shown in a similar way to Theorem 5.1 of \cite{Ho}.

\begin{theo}\label{3_conv of driv}
There exists an $\mathcal{H}_{\text{\rm KPZ}}^\k$-valued random variable $\mathbb{H}$ such that, for every $T>0$ and $p\ge1$,
$$
E\$\mathbb{H}\$_T^p<\infty,\quad
\lim_{\e\downarrow0}E\$\mathbb{H}^\e-\mathbb{H}\$_T^p=
\lim_{\e\downarrow0}E\$\tilde{\mathbb{H}}^\e-\mathbb{H}\$_T^p=0.
$$
In particular, both $h^\e=S_{\text{\rm KPZ}}(f(0),g(0),\mathbb{H}^\e)$ and $\tilde h^\e=S_{\text{\rm KPZ}}^\e(f(0),g(0),\tilde{\mathbb{H}}^\e)$ converge to $h=S_{\text{\rm KPZ}}(f(0),g(0),\mathbb{H})$ in probability 
as $\e\downarrow 0$ in $C([0,T_{\text{\rm sur}}),(\mathcal{C}^{\k\wedge(\mu+1)\wedge(2\mu+1)})^d)$.
\end{theo}

\subsection{Derivation of $c^\e A$}\label{3_section 3.3}

In Sections \ref{3_section 3.3}-\ref{3_section 3.5}, we compute the precise values of renormalization factors. First we consider the $\mathcal{H}_0$-component of the product $\partial_xH_\I\partial_xH_\I$.

Recall that $H_\I^\e=\tilde H_\I^\e$ is given by \eqref{3_eq:def of HI}. Since $p(t,x,y)=\mathcal{F}^{-1}(e^{-2\pi^2k^2t})(x-y)$, we have
\begin{align*}
\Pi_0^\bot H_\I^{\e,\a}(t,x)
&=\si_\b^\a\sum_{k\neq0}\int_{-\infty}^te^{2\pi ikx}e^{-2\pi^2k^2(t-s)}\fa(\e k)\xi^{\b,k}(ds).
\end{align*}
Therefore,
\begin{equation}  \label{3_eq:dh}
\partial_x H_\I^{\e,\a}(t,x)
 = \sum_{k\neq0} \int \mathcal{K}_\I^{\e,\a}(t,x)_{\b,k}(s) 
\xi^{\b,k}(ds) \in \mathcal{H}_1,
\end{equation}
where
\begin{equation}  \label{3_eq:4.f}
\mathcal{K}_\I^{\e,\a}(t,x)_{\b,k}(s) = \si_\b^\a e^{2\pi ik x} \fa(\e k) (2 \pi ik)1_{\{t\ge s\}} e^{-2\pi^2k^2(t-s)}.
\end{equation}
By Proposition \ref{3_prop:prod}, the $\mathcal{H}_2$-component of $(\partial_x H_\I^{\e,\a_1}\partial_x H_\I^{\e,\a_2})(t,x)$ is given by Wiener-It\^o integral with the kernel
\begin{align}  \label{3_eq:4.4}
&\mathcal{K}_\V^{\e,\a_1\a_2}(t,x)_{(\b_1,k_1),(\b_2,k_2)}(s_1,s_2) \\
\notag &:= \mathcal{K}_\I^{\e,\a_1}(t,x)_{\b_1,k_1}(s_1) \mathcal{K}_\I^{\e,\a_2}(t,x)_{\b_2,k_2}(s_2)  \\
\notag&\;=\si_{\b_1}^{\a_1}\si_{\b_2}^{\a_2}e^{2\pi i(k_1+k_2)x}\fa(\e k_1)\fa(\e k_2)\\
\notag&\qquad\times(2 \pi ik_1)1_{\{t\ge s_1\}} e^{-2\pi^2k_1^2(t-s_1)}(2 \pi ik_2)1_{\{t\ge s_2\}} e^{-2\pi^2k_2^2(t-s_2)},
\end{align}
while its $\mathcal{H}_0$-component is given by
\begin{align}  \label{3_eq:4.5}
&\int_\T\mathcal{K}_\I^{\e,\a_1}(t,x)_{\b,k}(s) \mathcal{K}_\I^{\e,\a_2}(t,x)_{\b,-k}(s)ds\\
&\notag = \sum_{\b=1}^d\si_\b^{\a_1} \si_\b^{\a_2} \sum_{k\neq0} \fa^2(\e k) \int 
(2\pi i k) (-2\pi i k) \big(  1_{\{t\ge s\}} e^{-2\pi^2k^2 (t-s)} \big)^2 ds \\
& = A^{\a_1\a_2} \sum_{k\neq0} \fa^2(\e k) \int_0^{\infty}  
4\pi^2 k^2 e^{-4\pi^2k^2s} ds   \notag   \\
& = A^{\a_1\a_2} \sum_{k\not=0} \fa^2(\e k)
\equiv c^\e A^{\a_1\a_2},   \notag
\end{align}
note that $\fa(\e k) = \fa(-\e k)$.

\subsection{Derivation of $C^\e$}

Recall that $H^\e_\Y$ and $\tilde H^\e_\Y$ are given by \eqref{3_eq:def of HY}. Now we introduce the kernels
$$
H_t(k):=1_{\{k\neq0,t>0\}}e^{-2\pi^2k^2t},\quad
h_t(k):=(2\pi ik)H_t(k).
$$
Then for the function $F(t,x)=e^{2\pi ikx}f(t)$, we have the formula
$$
\partial_xI(F)=e^{2\pi ikx}\int_\R h_{t-u}(k)f(u)du.
$$
The following convolution formula is useful, cf. the proof of Lemma 6.11 of \cite{HQ}.

\begin{lemm}\label{3_int hh H}
$$
\int_\R h_{t-u}(k)h_{s-u}(-k)du=H_{|t-s|}(k).
$$
\end{lemm}

\begin{proof}
Since the integral is over $u<t\wedge s$, the left hand side is rewritten as
\begin{align*}
&(2\pi ik)(-2\pi ik)\int_{-\infty}^{t\wedge s}e^{-2\pi^2k^2\{(t-u)+(s-u)\}}du\\
&=e^{-2\pi^2k^2(t+s-2u)}\big|_{u=-\infty}^{t\wedge s}=e^{-2\pi^2k^2|t-s|}=H_{|t-s|}(k).
\end{align*}
Note that $t+s-2(t\wedge s)=|t-s|$.
\end{proof}

The kernels of $\partial_x H_\Y^{\e,\a},\partial_x\tilde H_\Y^{\e,\a}\in \mathcal{H}_2$ are given by
\begin{align*}
&\mathcal{K}_\Y^{\e,\a}(t,x)_{(\b_1,k_1),(\b_2,k_2)}(s_1,s_2)\\
&=\tfrac12\Ga_{\ga_1\ga_2}^\a \partial_xI(\mathcal{K}_\V^{\e,\ga_1\ga_2}(\cdot,\cdot)_{(\b_1,k_1),(\b_2,k_2)}(s_1,s_2))(t,x)\\
&=\tfrac12 E_{\b_1\b_2}^\a e^{2\pi i(k_1+k_2)x} \fa(\e k_1)\fa(\e k_2)\int
h_{t-u}(k_1+k_2) h_{u-s_1}(k_1)h_{u-s_2}(k_2) du,
\end{align*}
and
\begin{align*}
&\tilde{\mathcal{K}}_\Y^{\e,\a}(t,x)_{(\b_1,k_1),(\b_2,k_2)}(s_1,s_2)\\
&=\tfrac12 E_{\b_1\b_2}^\a e^{2\pi i(k_1+k_2)x} \fa(\e k_1)\fa(\e k_2)\fa^2(\e(k_1+k_2))\\
&\quad\times\int
h_{t-u}(k_1+k_2) h_{u-s_1}(k_1)h_{u-s_2}(k_2) du,
\end{align*}
respectively. Here $E_{\b_1\b_2}^\a := \Ga_{\ga_1\ga_2}^\a \si_{\b_1}^{\ga_1}\si_{\b_2}^{\ga_2}$.

Now we compute two expectations
\begin{align*}
C^{\e,\a_1\a_2}&=E[(\partial_xH_\Y^{\e,\a_1}\partial_xH_\Y^{\e,\a_2})(t,x)],\\
\tilde C^{\e,\a_1\a_2}&=E[(\partial_x\tilde H_\Y^{\e,\a_1}\partial_x\tilde H_\Y^{\e,\a_2})(t,x)].
\end{align*}
By Proposition \ref{3_prop:prod},
\begin{align*}
C^{\e,\a_1\a_2}&=2\sum_{\b_1,\b_2}\sum_{k_1,k_2}\iint\mathcal{K}_\Y^{\e,\a_1}(t,x)_{(\b_1,k_1),(\b_2,k_2)}(s_1,s_2)\\
&\quad\times\mathcal{K}_\Y^{\e,\a_2}(t,x)_{(\b_1,-k_1),(\b_2,-k_2)}(s_1,s_2)ds_1ds_2,
\end{align*}
and $\tilde C^{\e,\a_1\a_2}$ is obtained by replacing $\mathcal{K}_\Y^\e$ by $\tilde{\mathcal K}_\Y^\e$.
The factor $2$ comes from the symmetry of the kernels.

\begin{lemm}\label{3_lem:C}  We have
\begin{align*}
&C^{\e,\b\ga} 
= \frac{F^{\b\ga}}{4\pi^2}
\sum_{k_1,k_2}\hspace{-7mm}\phantom{\sum}^{^\ast}
 \frac{\fa^2(\e k_1) \fa^2(\e k_2)}{k_1^2+k_1k_2+k_2^2},\\
&\tilde C^{\e,\b\ga} 
= \frac{F^{\b\ga}}{4\pi^2}
\sum_{k_1,k_2}\hspace{-7mm}\phantom{\sum}^{^\ast}
 \frac{\fa^2(\e k_1) \fa^2(\e k_2)\fa^4(\e (k_1+k_2))}{k_1^2+k_1k_2+k_2^2},
\end{align*}
where
$$
F^{\b\ga} = \sum_{\ga_1,\ga_2} E_{\ga_1 \ga_2}^\b E_{\ga_1 \ga_2}^\ga
=\sum_{\ga_1,\ga_2}\Ga_{\de_1\de_2}^\b\Ga_{\de_3\de_4}^\ga\si_{\ga_1}^{\de_1}\si_{\ga_2}^{\de_2}\si_{\ga_1}^{\de_3}\si_{\ga_2}^{\de_4}.
$$
Here $\sum^\ast$ means the sum over $k_1,k_2$ such that $k_1,k_2,k_1+k_2\neq0$.
\end{lemm}

\begin{rem}\label{3_rem:3.1}
When $d=1$, $C^{\epsilon,\beta\gamma}$ and $\tilde{C}^{\epsilon,\beta\gamma}$ coincide with $c^{\e,\Bo}$ and $\tilde c^{\e,\Bo}$, respectively, in \cite{Ho} with $E^\b_{\ga_1,\ga_2}=1$.
\end{rem}

\begin{rem}\label{rem:section3 graph and F}
The expression of the factor $F^{\b\ga}$ can be obtained by the following graphic rules. Each leaf of the graph ``\,\Bo" correspond to a label of a noise. When two noises are contracted, the two labels are equal. Each edge attached to a noise corresponds to the factor $\si_\b^\a$. Each vertex with the shape ``
\begin{tikzpicture}
\coordinate (A1) at (0,0);
\coordinate (A2) at ($0.6*(0.2,0.2)$);
\coordinate (A3) at ($0.6*(-0.2,0.2)$);
\coordinate (A4) at ($0.6*(0,-0.2)$);
\draw (A3)--(A1)--(A2);
\draw (A1)--(A4);
\end{tikzpicture}"
corresponds to the factor $\Ga_{\b\ga}^\a$. Indeed, we have
$$
\begin{tikzpicture}[baseline=20pt]
\coordinate (A1) at (0,0);
\coordinate (A2) at ($4*(-0.2,0.2)$);
\coordinate (A3) at ($4*(0.2,0.2)$);
\coordinate (A4) at ($4*(-0.3,0.4)$);
\coordinate (A5) at ($4*(-0.1,0.4)$);
\coordinate (A6) at ($4*(0.1,0.4)$);
\coordinate (A7) at ($4*(0.3,0.4)$);
\foreach \n in {4,5,6,7} \fill (A\n) circle (0.9pt);
\draw (A2)--(A1)--(A3);
\draw (A4)--(A2)--(A5);
\draw (A6)--(A3)--(A7);
\draw[dotted] (A4) to [out=30,in=150] (A6);
\draw[dotted] (A5) to [out=30,in=150] (A7);
\draw (A4) [fill=white] circle (4pt);
\node at (A4) {\tiny$\ga_1$};
\draw (A5) [fill=white] circle (4pt);
\node at (A5) {\tiny$\ga_2$};
\draw (A6) [fill=white] circle (4pt);
\node at (A6) {\tiny$\ga_1$};
\draw (A7) [fill=white] circle (4pt);
\node at (A7) {\tiny$\ga_2$};
\draw[white] ($(A2)!.5!(A1)$) [fill=white] circle (4pt);
\node at ($(A2)!.5!(A1)$) {\tiny$\b$};
\draw[white] ($(A3)!.5!(A1)$) [fill=white] circle (4pt);
\node at ($(A3)!.5!(A1)$) {\tiny$\ga$};
\draw[white] ($(A2)!.5!(A4)$) [fill=white] circle (4pt);
\node at ($(A2)!.5!(A4)$) {\tiny$\de_1$};
\draw[white] ($(A2)!.5!(A5)$) [fill=white] circle (4pt);
\node at ($(A2)!.5!(A5)$) {\tiny$\de_2$};
\draw[white] ($(A3)!.5!(A6)$) [fill=white] circle (4pt);
\node at ($(A3)!.5!(A6)$) {\tiny$\de_3$};
\draw[white] ($(A3)!.5!(A7)$) [fill=white] circle (4pt);
\node at ($(A3)!.5!(A7)$) {\tiny$\de_4$};
\end{tikzpicture}
=\sum_{\ga_1,\ga_2}\Ga_{\de_1\de_2}^\b\Ga_{\de_3\de_4}^\ga\si_{\ga_1}^{\de_1}\si_{\ga_2}^{\de_2}\si_{\ga_1}^{\de_3}\si_{\ga_2}^{\de_4}.
$$
\end{rem}

\begin{proof}
From Lemma \ref{3_int hh H},
\begin{align*}
C^{\e,\a_1\a_2}
&= \tfrac12\sum_{\b_1,k_1,\b_2,k_2} E_{\b_1 \b_2}^{\a_1} E_{\b_1 \b_2}^{\a_2} 
\fa^2(\e k_1) \fa^2(\e k_2) \\
&\quad\times\iint\!\!\!\iint
h_{t-u_1}(k_1+k_2) h_{u_1-s_1}(k_1)h_{u_1-s_2}(k_2)\\
& \quad\times h_{t-u_2}(-k_1-k_2) h_{u_2-s_1}(-k_1)h_{u_2-s_2}(-k_2) du_1du_2ds_1ds_2\\
&= \tfrac12F^{\a_1\a_2}\sum_{k_1,k_2\neq0}\fa^2(\e k_1) \fa^2(\e k_2) \\
&\quad\times 4\pi^2(k_1+k_2)^2\iint H_{t-u_1}(k_1+k_2)H_{t-u_2}(k_1+k_2)\\
&\quad\times H_{|u_1-u_2|}(k_1)H_{|u_1-u_2|}(k_2)du_1du_2.
\end{align*}
Note that the dependence in $x$ cancels.  Changing the variables
as $t-u_1=r_1, t-u_2=r_2$, the integral can be rewritten as
\begin{align}  \label{3_eq:4.10-b}
&\int_0^\infty  dr_1\int_0^\infty dr_2\ 
e^{-2\pi^2(k_1+k_2)^2(r_1+r_2)}e^{-2\pi^2(k_1^2+k_2^2)|r_1-r_2|}
\\
& =  \int_0^\infty dr_2\int_0^{r_2} dr_1\
e^{-2\pi^2(k_1+k_2)^2(r_1+r_2)}e^{-2\pi^2(k_1^2+k_2^2)(r_2-r_1)} \notag  \\
& \quad
+ (\text{a similar term with }k_1 \leftrightarrow k_2).
\notag
\end{align}
Since $k_1,k_2\not=0$, the first integral is equal to
\begin{align*}
&\int_0^\infty e^{-2\pi^2\{(k_1+k_2)^2+(k_1^2+k_2^2)\}r_2}dr_2
\int_0^{r_2} e^{-2\pi^2\{(k_1+k_2)^2-(k_1^2+k_2^2)\}r_1}dr_1\\
&=\frac{1}{2\pi^2\{(k_1+k_2)^2-(k_1^2+k_2^2)\}}\\
&\quad\times\int_0^\infty e^{-2\pi^2\{(k_1+k_2)^2+(k_1^2+k_2^2)\}r_2}
\big(1-e^{-2\pi^2\{(k_1+k_2)^2-(k_1^2+k_2^2)\}r_2}\big)dr_2\\
&=\frac{1}{4\pi^2k_1k_2}\int_0^\infty\big(e^{-2\pi^2\{(k_1+k_2)^2+(k_1^2+k_2^2)\}r_2}-e^{-4\pi^2(k_1+k_2)^2r_2}\big)dr_2\\
&=\frac{1}{4\pi^2k_1k_2}\bigg(\frac{1}{2\pi^2\{(k_1+k_2)^2+(k_1^2+k_2^2)\}}-\frac{1}{4\pi^2(k_1+k_2)^2}\bigg)\\
&=\frac{1}{4\pi^2k_1k_2}\frac{1}{2\pi^2}\frac{(k_1+k_2)^2-(k_1^2+k_2^2)}{2\{(k_1+k_2)^2+(k_1^2+k_2^2)\}(k_1+k_2)^2}\\
&=\frac{1}{4\pi^2k_1k_2}\frac{1}{2\pi^2}\frac{2k_1k_2}{4(k_1^2+k_1k_2+k_2^2)(k_1+k_2)^2}=\frac{1}{16\pi^4(k_1^2+k_1k_2+k_2^2)(k_1+k_2)^2}.
\end{align*}
By the symmetry under $k_1 \leftrightarrow k_2$, \eqref{3_eq:4.10-b} is equal to
$$
\frac{1}{8\pi^4(k_1^2+k_1k_2+k_2^2)(k_1+k_2)^2}.
$$
One can compute $\tilde C^{\e,\b\ga}$ similarly by noting that the kernel 
$\tilde{\mathcal{K}}_\Y^\e$ has an extra factor $\fa^2(\e(k_1+k_2))$ compared with
$\mathcal{K}_\Y^\e$.  This leads to the conclusion.
\end{proof}

\subsection{Derivation of $D^\e$}\label{3_section 3.5}

For the $\mathcal{H}_1$-component of $(\partial_x H_\Y^{\e,\a_1} \partial_x H_\I^{\e,\a_2})(t,x) \in \mathcal{H}_3 \oplus \mathcal{H}_1$, by Proposition \ref{3_prop:prod}, the kernel is given by
\begin{align*}
&\mathcal{K}_\Wc^{\e,\a_1\a_2}(t,x)_{\b,k}(s)\\
&=2\sum_{\b'=1}^d\sum_{k'\neq0}\int \mathcal{K}_\Y^{\e,\a_1}(t,x)_{(\b,k),(\b',k')}(s,s')\mathcal{K}_\I^{\e,\a_2}(t,x)_{\b',-k'}(s')ds'\\
&=e^{2\pi ikx}\fa(\e k)\sum_{\b'=1}^dE_{\b\b'}^{\a_1}\si_{\b'}^{\a_2}\sum_{k'\neq0}\fa^2(\e k')\\
&\quad\times\iint h_{t-u}(k+k')h_{u-s}(k)h_{u-s'}(k')h_{t-s'}(-k')duds'\\
&=e^{2\pi ikx}\fa(\e k)\sum_{\b'=1}^dE_{\b\b'}^{\a_1}\si_{\b'}^{\a_2}\sum_{k'\neq0}\fa^2(\e k')\int h_{t-u}(k+k')h_{u-s}(k)H_{t-u}(k')du.
\end{align*}
Note that $2$ in the first line comes from the symmetry of $\mathcal{K}_\Y^{\e,\a}(t,x)_{(\b_1,k_1),(\b_2,k_2)}(s_1,s_2)$ in $(\b_1,k_1,s_1)$ and $(\b_2,k_2,s_2)$.  We use Lemma \ref{3_int hh H} to have the last equality. Similarly, the $\mathcal{H}_1$-component of $(\partial_x \tilde H_\Y^{\e,\a_1}\partial_x H_\I^{\e,\a_2})(t,x)$ is given by the kernel
\begin{align*}
\tilde{\mathcal{K}}_\Wc^{\e,\a_1\a_2}(t,x)_{\b,k}(s)
&=e^{2\pi ikx}\fa(\e k)\sum_{\b'=1}^dE_{\b\b'}^{\a_1}\si_{\b'}^{\a_2}\sum_{k'\neq0}\fa^2(\e k')\fa^2(\e (k+k'))\\
&\quad\times\int h_{t-u}(k+k')h_{u-s}(k)H_{t-u}(k')du.
\end{align*}

Recall that $H_\W^{\e}$ and $\tilde H_\W^\e$ are given by \eqref{3_eq:def of HW}. For the $\mathcal{H}_1$-component of $\partial_xH_\W^{\e,\a}$ and $\partial_x\tilde H_\W^{\e,\a}$, the kernels are respectively given by
\begin{align*}
\mathcal{K}_\IWc^{\e,\a}(t,x)_{\b,k}(s)
&=\Ga_{\ga_1\ga_2}^\a\partial_xI(\mathcal{K}_\Wc^{\e,\ga_1\ga_2}(\cdot,\cdot)_{\b,k}(s))(t,x)\\
&=e^{2\pi ikx}\sum_{\b'=1}^d\Ga_{\ga_1\ga_2}^\a E_{\b\b'}^{\ga_1}\si_{\b'}^{\ga_2}\fa(\e k)\sum_{k'\neq0}\fa^2(\e k')\\
&\quad\times\iint h_{t-v}(k) h_{v-u}(k+k')h_{u-s}(k)H_{v-u}(k')dudv,
\end{align*}
and
\begin{align*}
\tilde{\mathcal{K}}_\IWc^{\e,\a}(t,x)_{\b,k}(s)
&=e^{2\pi ikx}\sum_{\b'=1}^d\Ga_{\ga_1\ga_2}^\a E_{\b\b'}^{\ga_1}\si_{\b'}^{\ga_2}\fa^3(\e k)\sum_{k'\neq0}\fa^2(\e k')\fa^2(\e(k+k'))\\
&\quad\times\iint h_{t-v}(k) h_{v-u}(k+k')h_{u-s}(k)H_{v-u}(k')dudv.
\end{align*}

Now we start to compute the expectations
\begin{align*}
D^{\e,\a_1\a_2}=E[(\partial_xH_\W^{\e,\a_1}\partial_xH_\I^{\e,\a_2})(t,x)],\\
\tilde D^{\e,\a_1\a_2}=E[(\partial_x\tilde H_\W^{\e,\a_1}\partial_xH_\I^{\e,\a_2})(t,x)].
\end{align*}
By Proposition \ref{3_prop:prod},  we have
\begin{align*}
D^{\e,\a_1\a_2}
&=\sum_{\b=1}^d\sum_{k\neq0}\int\mathcal{K}_\IWc^{\e,\a_1}(t,x)_{\b,k}(s)\mathcal{K}_\I^{\e,\a_2}(t,x)_{\b,-k}(s)ds,
\end{align*}
and $\tilde D^{\e,\a_1\a_2}$ is obtained by replacing $\mathcal K_\IWc^\e$ by $\tilde{\mathcal K}_\IWc^\e$.

\begin{lemm}\label{3_lem:D}  We have
\begin{align*}
D^{\e,\b\ga} 
&= -\frac{G^{\b\ga}}{4\pi^2}
\sum_{k_1,k_2}\hspace{-7mm}\phantom{\sum}^{^\ast}
 \frac{(k_1+k_2)\fa^2(\e k_1) \fa^2(\e k_2)}{k_1(k_1^2+k_1k_2+k_2^2)},\\
\tilde D^{\e,\b\ga} 
&= -\frac{G^{\b\ga}}{4\pi^2}
\sum_{k_1,k_2}\hspace{-7mm}\phantom{\sum}^{^\ast}
 \frac{(k_1+k_2)\fa^2(\e k_1) \fa^2(\e k_2)\fa^4(\e (k_1+k_2))}{k_1(k_1^2+k_1k_2+k_2^2)},
\end{align*}
where
$$
G^{\b\ga} = \sum_{\b_1,\b_2} \Ga_{\ga_1\ga_2}^{\b} E_{\b_1\b_2}^{\ga_1}\si_{\b_2}^{\ga_2}\si_{\b_1}^{\ga}
=\sum_{\b_1,\b_2}\Ga_{\ga_1\ga_2}^\b\Ga_{\de_1\de_2}^{\ga_1}\si_{\b_1}^{\de_1}\si_{\b_2}^{\de_2}\si_{\b_2}^{\ga_2}\si_{\b_1}^\ga.
$$
\end{lemm}

\begin{rem}  \label{3_rem:3.2}
When $d=1$, $D^{\epsilon,\beta\gamma}$ and $\tilde{D}^{\epsilon,\beta\gamma}$ coincide with $c^{\e,\D}$ and $\tilde c^{\e,\D}$, respectively, in \cite{Ho} with $G^{\b\ga}=1$.
\end{rem}

\begin{rem}\label{rem:section3 graph and G}
Similarly to Remark \ref{rem:section3 graph and F}, the expression of the factor $G^{\b\ga}$ can be obtained by the following graphic rules as follws.
$$
\begin{tikzpicture}[baseline=40pt]
\coordinate (A1) at (0,0);
\coordinate (A2) at ($4*(-0.1,0.2)$);
\coordinate (A3) at ($4*(0.1,0.2)$);
\coordinate (A4) at ($4*(-0.2,0.4)$);
\coordinate (A5) at ($4*(0,0.4)$);
\coordinate (A6) at ($4*(-0.3,0.6)$);
\coordinate (A7) at ($4*(-0.1,0.6)$);
\draw (A6)--(A1)--(A3);
\draw (A2)--(A5);
\draw (A4)--(A7);
\draw[dotted] (A7) to [out=0,in=90] (A5);
\draw[dotted] (A6)..controls($4*(-0.3,0.7)$)and($4*(0.2,0.8)$)..(A3);
\draw (A3) [fill=white] circle (4pt);
\node at (A3) {\tiny$\b_1$};
\draw (A5) [fill=white] circle (4pt);
\node at (A5) {\tiny$\b_2$};
\draw (A6) [fill=white] circle (4pt);
\node at (A6) {\tiny$\b_1$};
\draw (A7) [fill=white] circle (4pt);
\node at (A7) {\tiny$\b_2$};
\draw[white] ($(A2)!.5!(A1)$) [fill=white] circle (4pt);
\node at ($(A2)!.5!(A1)$) {\tiny$\b$};
\draw[white] ($(A3)!.5!(A1)$) [fill=white] circle (4pt);
\node at ($(A3)!.5!(A1)$) {\tiny$\ga$};
\draw[white] ($(A2)!.5!(A4)$) [fill=white] circle (4pt);
\node at ($(A2)!.5!(A4)$) {\tiny$\ga_1$};
\draw[white] ($(A2)!.5!(A5)$) [fill=white] circle (4pt);
\node at ($(A2)!.5!(A5)$) {\tiny$\ga_2$};
\draw[white] ($(A4)!.5!(A6)$) [fill=white] circle (4pt);
\node at ($(A4)!.5!(A6)$) {\tiny$\de_1$};
\draw[white] ($(A4)!.5!(A7)$) [fill=white] circle (4pt);
\node at ($(A4)!.5!(A7)$) {\tiny$\de_2$};
\end{tikzpicture}
=\sum_{\b_1,\b_2}\Ga_{\ga_1\ga_2}^\b\Ga_{\de_1\de_2}^{\ga_1}\si_{\b_1}^{\de_1}\si_{\b_2}^{\de_2}\si_{\b_2}^{\ga_2}\si_{\b_1}^\ga.
$$
\end{rem}

\begin{proof}
From Lemma \ref{3_int hh H},
\begin{align*}
D^{\e,\a_1\a_2}
&=\sum_{\b_1,\b_2}\sum_{k,k'\neq0}\Ga_{\ga_1\ga_2}^{\a_1} E_{\b_1\b_2}^{\ga_1}\si_{\b_2}^{\ga_2}\si_{\b_1}^{\a_2}\fa^2(\e k)\fa^2(\e k')\\
&\quad\times\iiint h_{t-v}(k) h_{v-u}(k+k')h_{u-s}(k)H_{v-u}(k')h_{t-s}(-k)dudvds\\
&=\sum_{k_1,k_2\neq0}G^{\a_1\a_2}\fa^2(\e k_1)\fa^2(\e k_2)(2\pi ik_1)\{2\pi i(k_1+k_2)\}\\
&\quad\times\int_{-\infty}^tdu\int_u^tdv\ H_{t-v}(k_1) H_{v-u}(k_1+k_2)H_{v-u}(k_2)H_{t-u}(k_1).
\end{align*}
Changing the variables as $t-u=u',t-v=v'$, the integral is computed as
\begin{align*}
&\int_0^\infty du'\int_0^{u'} dv'\ e^{-2\pi^2k_1^2v'}e^{-2\pi^2\{(k_1+k_2)^2+k_2^2\}(u'-v')}e^{-2\pi^2k_1^2u'}\\
&=\frac{1}{2\pi^2\{(k_1+k_2)^2+k_2^2-k_1^2\}}\\
&\quad\times\int_0^\infty du'\ e^{-2\pi^2\{(k_1+k_2)^2+k_1^2+k_2^2\}u'}\big(e^{2\pi^2\{(k_1+k_2)^2+k_2^2-k_1^2\}u'}-1\big)\\
&=\frac{1}{2\pi^2\{(k_1+k_2)^2+k_2^2-k_1^2\}}\int_0^\infty\big(e^{-4\pi^2k_1^2u'}-e^{-2\pi^2\{(k_1+k_2)^2+k_1^2+k_2^2\}u'}\big)du'\\
&=\frac{1}{2\pi^2\{(k_1+k_2)^2+k_2^2-k_1^2\}}\bigg(\frac{1}{4\pi^2k_1^2}-\frac{1}{2\pi^2\{(k_1+k_2)^2+k_1^2+k_2^2\}}\bigg)\\
&=\frac{1}{2\pi^2\{(k_1+k_2)^2+k_2^2-k_1^2\}}\frac{1}{2\pi^2}\frac{(k_1+k_2)^2-k_1^2+k_2^2}{2k_1^2\{(k_1+k_2)^2+k_1^2+k_2^2\}}\\
&=\frac{1}{16\pi^4k_1^2(k_1^2+k_1k_2+k_2^2)}.
\end{align*}
The computation of $\tilde D^{\e,\b\ga}$ is similar with two extra factors $\fa^2(\e(k_1+k_2))$.
This leads to the conclusion.
\end{proof}


\section{Renormalization factors under the trilinear condition \eqref{3_eq:1.2-s-hat}}\label{3_section 4}

In Lemmas \ref{3_lem:C} and \ref{3_lem:D}, we have already computed the renormalization factors
$$
B^{\e,\b\ga}=C^{\e,\b\ga}+2D^{\e,\b\ga},\quad
\tilde B^{\e,\b\ga}=\tilde C^{\e,\b\ga}+2\tilde D^{\e,\b\ga}
$$
with four renormalization factors given by
\begin{align*}
C^{\e,\b\ga} = F^{\b\ga} C^\e, \quad
D^{\e,\b\ga} = G^{\b\ga}D^\e, \quad
\tilde C^{\e,\b\ga} = F^{\b\ga} \tilde C^\e, \quad
\tilde D^{\e,\b\ga} = G^{\b\ga}\tilde D^\e,
\end{align*}
where $C^\e,\tilde C^\e,D^\e,\tilde D^\e$ depend only on $\fa$ and $\e$, and diverges as $O(-\log\e)$, while $F$ and $G$ are matrices determined from $\si$ and $\Ga$.
In this way, the renormalization factors are completely factorized into the products of two terms,
one determined from the scalar-valued KPZ equation as is pointed out in Remarks \ref{3_rem:3.1} and \ref{3_rem:3.2}
and the other from $\si$ and $\Ga$.

\begin{lemm}\label{lemm:expression of G and F}
The constants
$G^{\b\ga}$ and $F^{\b\ga}$ are rewritten as
\begin{align*}
G^{\b\ga} = \si_{\a_1}^\b\si_{\a_2}^\ga\hat\Ga_{\b_1\b_2}^{\a_1}\hat\Ga_{\a_2\b_2}^{\b_1},\quad
F^{\b\ga} = \si_{\a_1}^\b\si_{\a_2}^\ga\hat\Ga_{\b_1\b_2}^{\a_1}\hat\Ga_{\b_1\b_2}^{\a_2}, 
\end{align*}
respectively. Here the sums $\sum$ over $(\a_1,\a_2,\b_1,\b_2)$ are omitted. Moreover, the equality $F^{\b\ga}=G^{\b\ga}$ holds for every $(\b,\ga)$ if and only if the trilinear condition \eqref{3_eq:1.2-s-hat} holds. 
\end{lemm}

\begin{proof}
In what follows, summation symbols $\sum$ over the repeated indices are omitted. Noting $\Ga_{\ga_1\ga_2}^\a\si_{\b_1}^{\ga_1}\si_{\b_2}^{\ga_2}=\si_\b^\a\hat\Ga_{\b_1\b_2}^\b$, these constants are easily computed as
\begin{align*}
& G^{\b\ga} = \Ga_{\ga_1\ga_2}^\b(\si_{\a_2}^{\ga_1}\hat\Ga_{\b_1\b_2}^{\a_2})\si_{\b_2}^{\ga_2}\si_{\b_1}^\ga=
\si_{\a_1}^\b\si_{\b_1}^\ga\hat\Ga_{\a_2\b_2}^{\a_1}\hat\Ga_{\b_1\b_2}^{\a_2}
=\si_{\a_1}^\b\si_{\a_2}^\ga\hat\Ga_{\b_1\b_2}^{\a_1}\hat\Ga_{\a_2\b_2}^{\b_1}, \\
& F^{\b\ga} = (\si_{\a_1}^\b\hat\Ga_{\ga_1\ga_2}^{\a_1})(\si_{\a_2}^\ga\hat\Ga_{\ga_1\ga_2}^{\a_2})
= \si_{\a_1}^\b\si_{\a_2}^\ga\hat\Ga_{\ga_1\ga_2}^{\a_1}\hat\Ga_{\ga_1\ga_2}^{\a_2}
= \si_{\a_1}^\b\si_{\a_2}^\ga\hat\Ga_{\b_1\b_2}^{\a_1}\hat\Ga_{\b_1\b_2}^{\a_2}.
\end{align*}
Hence if we assume the trilinear condition \eqref{3_eq:1.2-s-hat}, we have
\begin{align*}
F^{\b\ga} =G^{\b\ga} =\si_{\a_1}^\b\si_{\a_2}^\ga\hat\Ga_{\b_1\b_2}^{\a_1}\hat\Ga_{\b_1\b_2}^{\a_2}.
\end{align*}
Conversely,
$$
G^{\b\ga}-F^{\b\ga} \equiv
\si_{\a_1}^\b\si_{\a_2}^\ga \hat\Ga_{\b_1\b_2}^{\a_1}\big(\hat\Ga_{\a_2\b_2}^{\b_1} 
- \hat\Ga_{\b_1\b_2}^{\a_2} \big)  =0
$$
is equivalent to
$$
\hat\Ga_{\b_1\b_2}^\b\big(\hat\Ga_{\ga\b_2}^{\b_1} 
- \hat\Ga_{\b_1\b_2}^\ga \big)=0,\quad1\le\b,\ga\le d,
$$
because $\si$ is invertible. Taking $\b=\ga$ and summing them over $\b$, we have
\begin{align}\label{3_F=G tri 1}
\sum_{\b,\b_1,\b_2}\hat\Ga_{\b_1\b_2}^\b\big(\hat\Ga_{\b\b_2}^{\b_1} 
- \hat\Ga_{\b_1\b_2}^\b \big)=0.
\end{align}
Replacing the role of variables $\b$ and $\b_1$ in \eqref{3_F=G tri 1}, we have
\begin{align}\label{3_F=G tri 2}
\sum_{\b,\b_1,\b_2}\hat\Ga_{\b\b_2}^{\b_1}\big(\hat\Ga_{\b\b_2}^{\b_1} 
- \hat\Ga_{\b_1\b_2}^\b \big)=0.
\end{align}
Taking the difference of \eqref{3_F=G tri 1} and \eqref{3_F=G tri 2}, we have
$$
\sum_{\b,\b_1,\b_2}\big(\hat\Ga_{\b\b_2}^{\b_1} 
- \hat\Ga_{\b_1\b_2}^{\b} \big)^2=0,
$$
which yields $\hat\Ga_{\b\b_2}^{\b_1} = \hat\Ga_{\b_1\b_2}^{\b}$ for every $(\b,\b_1,\b_2)$.
\end{proof}

\begin{proof}[{Proof of Theorem \ref{3_thm:1.2}}]
A computation for the scalar-valued case made in Proposition 5.32 of \cite{Ho} shows
\begin{align}\label{eq:cancellation between C and D}
\tilde C^{\e} + 2\tilde D^{\e}=0, \quad C^{\e} + 2D^{\e} = -\tfrac1{12} + O(\e)
\end{align}
(see also Lemma 6.5 of \cite{Hairer13}), and this implies that all components $B^{\e,\b\ga}$ and $\tilde B^{\e,\b\ga}$ behave as $O(1)$ if and only if 
$F=G$ as matrices, which is equivalent to the condition \eqref{3_eq:1.2-s-hat} by Lemma \ref{lemm:expression of G and F}.

Let $h_{B=0}^\e$ and $\tilde h_{\tilde B=0}^\e$ be the solutions of two KPZ approximating equations 
\eqref{3_eq:1-2} and \eqref{3_eq:1} with $B^{\e,\b\ga}, \tilde B^{\e,\b\ga}=0$, which are actually the shifts
$$
h^{\e,\a}_{B=0} = h^{\e,\a}+ \tfrac{t}2 \Ga_{\b\ga}^\a B^{\e,\b\ga},\quad
\tilde h^{\e,\a}_{\tilde B=0} = \tilde h^{\e,\a}+ \tfrac{t}2 \Ga_{\b\ga}^\a \tilde B^{\e,\b\ga}
$$
of the solutions $h^\e$ and $\tilde h^\e$ of \eqref{3_eq:1-2} and \eqref{3_eq:1}, respectively. Both of them converge because $B^{\e,\b\ga},\tilde B^{\e,\b\ga}=O(1)$ when \eqref{3_eq:1.2-s-hat} holds. Let $h_{B=0}$ and $\tilde{h}_{B=0}$ be the respective limits. The difference
$$
\tilde{h}^{\e,\a}_{\tilde B=0}- h^{\e,\a}_{B=0} 
 = \big(\tilde{h}^{\e,\a} - h^{\e,\a}\big)
+ \tfrac{t}2 \Ga_{\b\ga}^\a \big(\tilde B^{\e,\b\ga} - B^{\e,\b\ga}\big),
$$
converges because $\tilde{h}^{\e,\a} - h^{\e,\a}\to 0$ by Theorem \ref{3_thm:1.1}-(2). Furthermore, noting \eqref{eq:cancellation between C and D}, we have in the limit
$$
\tilde h^\a_{\tilde B=0}(t,x) = h^\a_{B=0}(t,x) + c^\a t,
$$
where
$$
c^\a
: = \frac1{24} \Ga^\a_{\b\ga} F^{\b\ga} 
= \frac1{24} \sum_{\b_1,\b_2} \Ga^\a_{\b\ga} \si_{\a_1}^\b\si_{\a_2}^\ga\hat\Ga_{\b_1\b_2}^{\a_1}\hat\Ga_{\b_1\b_2}^{\a_2}
= \frac1{24} \sum_{\b_1,\b_2} \si_\b^\a \hat\Ga_{\a_1\a_2}^\b 
 \hat\Ga_{\b_1\b_2}^{\a_1} \hat\Ga_{\b_1\b_2}^{\a_2}.
$$
\end{proof}


\section{Global existence for a.e.-initial values under the stationary measure}\label{3_section 5}

When $d=1$, the global-in-time existence of the solution of the KPZ equation was obtained by Gubinelli and Perkowski \cite{GP}, using the Cole-Hopf transform. In the multi-component case, however, such transform does not work in general, so that the global existence is non-trivial. In this section, by similar arguments to Da Prato and Debussche \cite{DPD}, we show the global existence for initial values sampled from the invariant measure of \eqref{3_eq:KPZ}, under the trilinear condition \eqref{3_eq:1.2-s-hat}.

Precisely, the process which has the invariant measure is the derivative $u=\partial_xh$, which solves the coupled stochastic Burgers equation
\begin{align}\label{3_eq:sbe}
\partial_tu^\a=\tfrac12\partial_x^2u^\a+\tfrac12\Ga_{\b\ga}^\a\partial_x(u^\b u^\ga)+\si_\b^\a\partial_x\xi^\b.
\end{align}
We can apply the paracontrolled calculus to \eqref{3_eq:sbe} and construct a well-posed solution map similarly to the coupled KPZ equation. Indeed, these two schemes are equivalent. If $h$ solves \eqref{3_eq:KPZ}, then $u=\partial_xh$ solves \eqref{3_eq:sbe}. Conversely, the solution $\hat h$ of
$$
\partial_t\hat h^\a=\tfrac12\partial_x^2\hat h^\a+\tfrac12\Ga_{\b\ga}^\a u^\b u^\ga+\si_\b^\a\xi^\b
$$
coincides with the original $h$. Hence the global existence of $u$ is equivalent to that of $h$. The equation \eqref{3_eq:sbe} has the Gaussian invariant measure $\mu_A$, under the condition \eqref{3_eq:1.2-s-hat}. As we will see, this implies the global existence of $u$ starting from a.e.-initial values under the invariant measure. We will justify the above arguments in this section.

From now we consider the space of zero mean functions denoted by
$$
\mathcal{C}_0^\a=\{u\in\mathcal{C}^\a\,;\ \int_\T u(x)dx=0\}.
$$

\subsection{Relation between KPZ equation and stochastic Burgers equation}

Construction of the solution map for \eqref{3_eq:sbe} is parallel to that for the coupled KPZ equation.
As in Section \ref{3_section 2}, although we eventually take $a=1$, considering the formal expansion $u^\a=\sum_{k=0}^\infty a^ku_k^\a$ of the solution of
$$
\mathcal{L}u^\a=\frac a2\Ga_{\b\ga}^\a\partial_x(u^\b u^\ga)+\si_\b^\a\partial_x\xi^\b,
$$
we obtain the identities
\begin{equation*}
\begin{aligned}
\mathcal{L} u_0^\a &= \si^\a_\b \partial_x\xi^\b, \\
\mathcal{L} u_1^\a &= \tfrac12 \Ga_{\b\ga}^\a 
\partial_x (u^\b_0 u^\ga_0), \\
\mathcal{L} u_2^\a &= \Ga_{\b\ga}^\a 
\partial_x (u^\b_1 u^\ga_0), \\
\mathcal{L} u_3^\a &= 
\tfrac12 \Ga_{\b\ga}^\a 
\partial_x (u^\b_1 u^\ga_1)
+ \Ga_{\b\ga}^\a 
\partial_x (u^\b_2 u^\ga_0).
\end{aligned}
\end{equation*}
We denote $u_0,u_1,u_2$ with stationary initial values by $U_\I,U_\Y,U_\W$, respectively. To define $\mathcal{L}u_3$, we introduce
$$
U_\Bo^{\b\ga}=\tfrac12 \partial_x(U_\Y^\b U_\Y^\ga),\quad
U_\D^{\b\ga}=\partial_x(U_\W^\b \rs U_\I^\ga).
$$
After defining these objects, \eqref{3_eq:sbe} for $u=U_\I+U_\Y+U_\W+u_{\ge3}$ can be rewritten as
\begin{align*}
\mathcal{L}u_{\ge3}^\a=\Psi^\a+\mathcal{L}u_3^\a,
\end{align*}
where
\begin{align*}
\Psi^\a & = \Ga_{\b\ga}^\a \partial_x (u^\b_{\ge 3} U^\ga_\I)
+ \Ga_{\b\ga}^\a \partial_x (U^\b_\W +u^\b_{\ge 3})
U^\ga_\Y+ \tfrac12 \Ga_{\b\ga}^\a \partial_x (U^\b_\W +u^\b_{\ge 3})
(U^\ga_\W +u^\ga_{\ge 3}).
\end{align*}
The term $u_{\ge3}^\b U_\I^\ga$ is still ill-posed. To make sense, we divide $u_{\ge3}=v+w$, which solve
\begin{equation}\label{3_sys:sbe}
\begin{aligned}
\mathcal{L}v^\a&=\Ga_{\b\ga}^\a\partial_x\{(U_\W^\b+v^\b+w^\b)\pl U_\I^\ga\},\\
\mathcal{L}w^\a&=\Ga_{\b\ga}^\a\partial_x\{(U_\W^\b+v^\b+w^\b)(\rs+\pr) U_\I^\ga\}+\text{other terms},
\end{aligned}
\end{equation}
respectively. The only remaining problem is to give the definition of $v^\b\rs U_\I^\ga$. Introducing $U_\K^\a$ as a stationary solution of $\mathcal{L}U_\K^\a=\partial_xU_\I^\a$, $v^\a$ has the form
$$
v^\a=\Ga_{\b\ga}^\a(U_\W^\b+v^\b+w^\b)\pl U_\K^\ga+\text{regular terms}.
$$
Thus, if $U_\Co^{\b\ga}=U_\K^\b\rs U_\I^\ga$ is given a priori, one can define the term $v^\b\rs U_\I^\ga$.

We summarize these arguments. Fix $\k\in (\frac13,\tfrac12)$. The driver of the coupled stochastic Burgers
equation \eqref{3_eq:sbe} is the element ${\mathbb U}$ of the form
\begin{align*}
&{\mathbb U} :=(U_\I,U_\Y,U_\W,U_\Bo,
U_\D,U_\K,U_\Co)\\
&\in C([0,T],(\mathcal{C}_0^{\k-1})^d)
\times C([0,T],(\mathcal{C}_0^{2\k-1})^d)
\times \{ C([0,T],(\mathcal{C}_0^{\k})^d) \cap C^{1/4}([0,T],(\mathcal{C}_0^{\k-1/2})^d)\} \\
&\quad\times C([0,T],(\mathcal{C}_0^{2\k-2})^{d^2}) \times C([0,T],(\mathcal{C}_0^{2\k-2})^{d^2})
\times C([0,T],(\mathcal{C}_0^{\k})^d) \times C([0,T],(\mathcal{C}^{2\k-1})^{d^2}),
\end{align*}
which satisfies $\mathcal{L}U_\K=\partial_xU_\I$. We denote by $\mathcal{U}_{\text{CSB}}^\k$ the class of all drivers. We write $\$\mathbb{U}\$_T$ for the product norm on the above space. Comparing with $\mathcal{H}_{\text{KPZ}}^\k$, note that
\begin{equation}\label{3_eq:5.2}
\begin{aligned}
&U_\circ=\partial_xH_\circ,\quad (\circ=\I,\Y,\W,\Bo,\D,\K),\\
&U_\Co=H_\Co.
\end{aligned}
\end{equation}
Now we can prove a similar result to Theorem \ref{3_thm:3}. Fix $\la\in(\tfrac13,\k)$ and $\mu\in(-\la,\la]$. For a $\mathcal{D}'(\T,\R^d)$-valued functions $v=(v^\a)_{\a=1}^d,w=(w^\a)_{\a=1}^d$ on $[0,T]$, we write $(v,w)\in\mathcal{D}_{\text{CSB}}^{\la,\mu}([0,T])$ if 
\begin{align*}
&\|(v,w)\|_{\mathcal{D}_{\text{CSB}}^{\la,\mu}([0,T])}:=\\
&\sup_{t\in[0,T]}t^{\frac{\la-\mu}2}\|v(t)\|_{(\mathcal{C}_0^{\la})^d}
+\sup_{t\in[0,1]}\|v(t)\|_{(\mathcal{C}_0^{\mu})^d}
+\sup_{s<t\in[0,T]}s^{\frac{\la-\mu}2}\frac{\|v(t)-v(s)\|_{(\mathcal{C}_0^{\la-1/2})^d}}{|t-s|^{1/4}}\\
&+\sup_{t\in[0,T]}t^{\la-\mu}\|w(t)\|_{(\mathcal{C}_0^{2\la})^d}
+\sup_{t\in[0,1]}\|w(t)\|_{(\mathcal{C}_0^{2\mu})^d}
+\sup_{s<t\in[0,T]}s^{\la-\mu}\frac{\|w(t)-w(s)\|_{(\mathcal{C}_0^{2\la-1/2})^d}}{|t-s|^{1/4}}
\end{align*}
is finite. For every initial value $(v(0),w(0))\in(\mathcal{C}_0^{\mu})^d\times(\mathcal{C}_0^{2\mu})^d$, the system \eqref{3_sys:sbe} admits a unique local solution $(v,w)\in\mathcal{D}_{\text{CSB}}^{\la,\mu}$. Denote by
$
u=U_\I+U_\Y+U_\W+v+w\equiv S_{\text{CSB}}(v(0),w(0),\mathbb{U})
$
the maximal solution up to the survival time.

From the constructions, we see that
$$
\partial_xS_{\text{KPZ}}(f(0),g(0),\mathbb{H})=S_{\text{CSB}}(\partial_xf(0),\partial_xg(0),\mathbb{U}),
$$
where $\mathbb{U}$ satisfies \eqref{3_eq:5.2}.  The problem is to restore the solution map $S_{\text{KPZ}}$ from $S_{\text{CSB}}$.
Since the right hand sides of \eqref{3_eq:fg} depend only on the derivatives of $f$ and $g$, we can write
\begin{equation}\label{3_eq:5.3}
\begin{aligned}
\mathcal{L} f^\a &= \Ga_{\b\ga}^\a 
(U_\W^\b+v^\b+w^\b)\pl
U^\ga_\I, \\
\mathcal{L} g^\a & = \Ga_{\b\ga}^\a 
(U_\W^\b+v^\b+w^\b)(\rs+\pr)
U^\ga_\I+\text{other terms}.
\end{aligned}
\end{equation}
Conversely let $f,g$ be the solutions of \eqref{3_eq:5.3} with initial values $f(0),g(0)$. Then $f,g$ should satisfy $\partial_xf=v$ and $\partial_xg=w$ by uniqueness of the solution of \eqref{3_sys:sbe}. Inserting these relations into \eqref{3_eq:5.3}, we see that $f,g$ satisfy \eqref{3_eq:fg}. Hence $(f,g)$ is the solution of the original KPZ equation. In this way, $S_{\text{KPZ}}$ can be
recovered from $S_{\text{CSB}}$.  To sum up, we have the following equivalence.

\begin{theo}\label{3_kpz equiv csb}
Assume that $\mathbb{H}$ and $\mathbb{U}$ are related by \eqref{3_eq:5.2}. Let $T_{\text{\rm sur}}^{\text{\rm KPZ}}$ and $T_{\text{\rm sur}}^{\text{\rm CSB}}$ be the survival times of the solutions of \eqref{3_eq:fg} and \eqref{3_sys:sbe}, respectively. Then we have $T_{\text{\rm sur}}^{\text{\rm KPZ}}=T_{\text{\rm sur}}^{\text{\rm CSB}}$ and
$$
\partial_xS_{\text{\rm KPZ}}(f(0),g(0),\mathbb{H})=S_{\text{\rm CSB}}(\partial_xf(0),\partial_xg(0),\mathbb{U}).
$$
\end{theo}

We constructed an $\mathcal{H}_{\text{KPZ}}^\k$-valued random variable $\mathbb{H}$ from space-time white noise $\xi$, in Section \ref{3_section 3}. The relation \eqref{3_eq:5.2} determines a $\mathcal{U}_{\text{CSB}}^\k$-valued random variable $\mathbb{U}$. Note that renormalization factors vanish because we take the derivative $\partial_x$. In the following sections, we study the probabilistic properties of the solution $u=S_{\text{CSB}}(v(0),w(0),\mathbb{U})$.

\subsection{Gaussian stationary measure of the OU process}

Let $\{u^{\a,k}\}_{\a\in\{1,\ldots,d\},k\neq0}$ be the family of centered complex Gaussian variables such that $u^{\a,-k}=\overline{u^{\a,k}}$ and has covariance
$$
E[u^{\a,k}u^{\b,l}]=A^{\a\b}1_{\{k+l=0\}}.
$$
Denote the distribution on $\mathcal{D}'(\T,\R^d)$ of $u^\a(x)=\sum_ku^{\a,k}e^{2\pi ikx}$ by $\mu_A$. Indeed, $\mu_A$ has a support on $(\mathcal{C}_0^{-1/2-})^d$.

\begin{lemm}
Let $\k<\tfrac12$. For $\mu_A$-a.e. $u\in\mathcal{D}'$, we have $u\in(\mathcal{C}_0^{\k-1})^d$.
\end{lemm}

\begin{proof}
Let $\zeta\in\R$ and $p\ge1$. Computing $L^{2p}(\Omega,P)$-norm of $\|u^\a\|_{\mathcal{B}_{2p,2p}^\zeta}$ as
\begin{align*}
E\|u^\a\|_{\mathcal{B}_{2p,2p}^\zeta}^{2p}
=\sum_j2^{\zeta j\cdot2p}E\|\triangle_ju^\a\|_{L^{2p}(\T)}^{2p}
=\sum_j2^{\zeta j\cdot2p}\int_\T E|\triangle_ju^\a(x)|^{2p}dx,
\end{align*}
where $\triangle_j=\rho(2^{-j}D)$ is a Littlewood-Paley projection (see Section 2.2 of \cite{Ho}), from the hypercontractivity of Wiener chaos, we have
\begin{align*}
E|\triangle_ju^\a(x)|^{2p}\lesssim(E|\triangle_ju^\a(x)|^2)^p.
\end{align*}
Since we have
\begin{align*}
E|\triangle_ju^\a(x)|^2
=\sum_k\rho_j(k)^2E[u^{\a,k}u^{\a,-k}]=A^{\a\a}\sum_k\rho_j(k)^2\lesssim 2^j,
\end{align*}
independently of $x$, for every $\zeta<-\tfrac12$
\begin{align*}
E\|u^\a\|_{\mathcal{B}_{2p,2p}^\zeta}^{2p}\lesssim\sum_j2^{(2\zeta+1)jp}<\infty.
\end{align*}
Since $\mathcal{B}_{2p,2p}^\zeta\subset\mathcal{B}_{\infty,\infty}^{\zeta-1/(2p)}$ by Besov embedding, 
see Theorem 2.71 of \cite{BCD} or Proposition 2.2 of \cite{Ho}, we have the conclusion for sufficiently large $p$.
\end{proof}

Though it is well-known, we show that the Ornstein-Uhlenbeck process $u$ determined by
\begin{align}\label{3_eq:ou}
\mathcal{L}u^\a=\si_\b^\a\partial_x\xi^\b
\end{align}
has an invariant measure $\mu_A$. Taking Fourier transform $u_0^{\a,k}=\hat u_0^\a(k)$, we can solve it as
$$
u^{\a,k}(t)=e^{-2\pi^2k^2t}u^{\a,k}(0)-\si_\b^\a(2\pi ik)\int_0^te^{-2\pi^2k^2(t-s)}\xi^{\b,k}(s)ds.
$$
For a given $u(0)$, $\{u^{\a,k}(t)\}_{\a,k}$ is a Gaussian family. If $u(0)\sim\mu_A$, i.e., $u(0)$ is
distributed under $\mu_A$, and if $u(0)$ is independent of $\xi$, we see that $u^{\a,k}$ has mean zero and covariance
\begin{align*}
E[u^{\a,k}(t)u^{\b,l}(t)]&=e^{-2\pi^2(k^2+l^2)t}E[u^{\a,k}(0)u^{\b,k}(0)]\\
&\quad+\si_{\ga_1}^\a\si_{\ga_2}^\b(2\pi ik)(2\pi il)\int_0^te^{-4\pi^2k^2(t-s)}\delta^{\ga_1\ga_2}1_{\{k+l=0\}}ds\\
&=A^{\a\b}1_{\{k+l=0\}}\bigg(e^{-4\pi^2k^2t}+4\pi^2k^2\int_0^te^{-4\pi^2k^2(t-s)}ds\bigg)\\
&=A^{\a\b}1_{\{k+l=0\}}.
\end{align*}
Hence $u(t)\sim\mu_A$ for every $t>0$.

\subsection{Galerkin approximation}\label{subsection:galerkin approximation}

For $N\in\mathbb{N}$, we consider the approximation
\begin{align}\label{3_galerkin:stationary}
\partial_tu^{N,\alpha}=\tfrac{1}{2}\partial_x^2u^{N,\alpha}+F_N^\alpha(u^N)+\sigma_\beta^\alpha\partial_x\xi^\beta,
\end{align}
of the equation \eqref{3_eq:sbe}, where
\begin{align}\label{galerkin:def of F^N}
F_N^\alpha(u^N)=\tfrac{1}{2}\Gamma_{\beta\gamma}^\alpha\partial_xP_N(P_Nu^{N,\b} P_Nu^{N,\ga}),
\end{align}
and $P_N=\psi(N^{-1}D)$ is the Fourier multiplier defined by an even cut-off function $\psi\in C_0^\infty(\R)$ taking values in $[0,1]$ and supported in the interval $[-1,1]$. Since $F_N$ depends on finitely many Fourier components of $u^N$, the equation \eqref{3_galerkin:stationary} is well-posed.

The equation \eqref{3_galerkin:stationary} is formally equivalent to the (spatial derivative of the) approximating equation \eqref{3_eq:1}. Indeed, for the solution $\tilde{u}^\epsilon$ of such equation, we would have that $\bar{u}^\epsilon:=\varphi^{-1}(\epsilon D)\tilde{u}^\epsilon\equiv\tilde{u}^\epsilon(*\eta^\epsilon)^{-1}$ solves
$$
\partial_t\bar{u}^{\epsilon,\alpha}=\tfrac12\partial_x^2\bar{u}^{\epsilon,\alpha}+\tfrac12\Gamma_{\beta\gamma}^\alpha\partial_x((\bar{u}^{\epsilon,\beta}*\eta^\epsilon)(\bar{u}^{\epsilon,\gamma}*\eta^\epsilon))*\eta^\epsilon+\sigma_\beta^\alpha\partial_x\xi^\beta.
$$
Here $\varphi^{-1}(\epsilon D)$ is an inverse operator of the convolution $*\eta^\epsilon=\varphi(\epsilon D)$ defined in a finite dimensional subspace of $\mathcal{D}'(\T)$. Then \eqref{3_galerkin:stationary} is obtained by setting $u^N=\bar{u}^\epsilon$ and $P_N=*\eta^\epsilon$. Since $\tilde{u}^\epsilon$ has an invariant measure $\mu_A^\epsilon$, which is the distribution of the derivative of the $d$-dimensional periodic and smeared Brownian motion $(\partial_x\sigma B*\eta_\epsilon)_{x\in\T}$, $\bar{u}^\epsilon$ should admit $\mu_A$ as an invariant measure.

Unlike the usual Galerkin approximation, we use the operator $P_N$ rather than Fourier cut-off $\Pi_N=1_{[-N,N]}(D)$. This is because $P_N$ has the approximating properties
$$
\|P_Nu\|_{\mathcal{C}^\alpha}\lesssim\|u\|_{\mathcal{C}^\alpha},\quad
\|P_Nu-u\|_{\mathcal{C}^{\alpha-\delta}}\lesssim N^{-\delta}\|u\|_{\mathcal{C}^\alpha},
$$
for $\a\in\R$ and $\de\in[0,2]$; see Lemma A.5 of \cite{GIP}, or Lemmas 2.4 and 2.5 of \cite{Ho}.  
We can construct the solution map $u^N=S_{\text{CSB}}^N(v(0),w(0),\mathbb{U}^N)$ corresponding to 
\eqref{3_galerkin:stationary}, where $\mathbb{U}^N$ is defined by the stationary solutions of
\begin{equation}\label{3_eq:UI-UK}
\begin{aligned}
\mathcal{L} U_\I^{N,\a} &= \si^\a_\b \partial_x\xi^\b, \\
\mathcal{L} U_\Y^{N,\a} &= \tfrac12 \Ga_{\b\ga}^\a 
\partial_x P_N(P_NU^{N,\b}_\I P_NU^{N,\ga}_\I), \\
\mathcal{L} U_\W^{N,\a} &= \Ga_{\b\ga}^\a 
\partial_x P_N(P_NU^{N,\b}_\Y P_NU^{N,\ga}_\I),\\
\mathcal{L} U_\K^{N,\a} &= \partial_xP_N^2U_\I^{N,\a}
\end{aligned}
\end{equation}
and products
\begin{equation}\label{3_eq:UB-UC}
\begin{aligned}
U_\Bo^{N,\b\ga}&=\tfrac12 \partial_x(P_NU_\Y^{N,\b}P_NU_\Y^{N,\ga}),\\
U_\D^{N,\b\ga}&=\partial_x(P_NU_\W^{N,\b}\rs P_NU_\I^{N,\ga}),\\
U_\Co^{N,\b\ga}&=P_NU_\K^{N,\b}\rs P_NU_\I^{N,\ga}.
\end{aligned}
\end{equation}
From the approximating properties of $P_N$, we have that $S_{\text{CSB}}^N\to S_{\text{CSB}}$ similarly to the approximation $S_{\text{KPZ}}^\e$.

\begin{theo}\label{3_galerkin:theo}
If $(v^N(0),w^N(0)) \to (v(0),w(0))$ in $(\mathcal{C}_0^{\mu})^d\times(\mathcal{C}_0^{2\mu})^d$ and 
${\mathbb U}^N \to {\mathbb U}$ in $\mathcal{U}_{\rm CSB}^\k$, 
then we have the convergence $S_{\text{\rm CSB}}^N(v^N(0),w^N(0),{\mathbb U}^N)
\to S_{\text{\rm CSB}}(v(0),w(0),{\mathbb U})$ in $C([0,T_{\text{\rm sur}}),(\mathcal{C}_0^{(\k-1)\wedge\mu\wedge2\mu})^d)$.
\end{theo}

If we define $\mathbb{U}^N$ by \eqref{3_eq:UI-UK}-\eqref{3_eq:UB-UC} from the space-time white noise $\xi$, then $S_{\text{CSB}}^N(v(0),w(0),\mathbb{U}^N)$ coincides with the strong solution of \eqref{3_galerkin:stationary} with initial value $(U_\I^N+U_\Y^N+U_\W^N+v+w)(0)$.

Our goal is to show that $\mu_A$ is invariant under $(u^N)$ if the trilinear condition \eqref{3_eq:1.2-s-hat} holds;
see Proposition \ref{3_prop:muA is inv under uN} below.
Let $u_{\text{OU}}$ be the solution of \eqref{3_eq:ou} with initial value $u^N(0)$. Obviously, the solution of \eqref{3_galerkin:stationary} is given by $u^N=\Pi_N^\bot u_{\text{OU}}+U^N$, where $\Pi_N^\bot:=1-\Pi_N$ and $U^N$ solves the finite dimensional SDE
\begin{align}\label{3_eq:finite galerkin}
\partial_tU^{N,\alpha}=\tfrac{1}{2}\partial_x^2U^{N,\alpha}+F_N^\alpha(U^N)+\sigma_\beta^\alpha\partial_x\Pi_N\xi^\beta
\end{align}
with $U^N(0)=\Pi_Nu^N(0)$. If $u^N(0)\sim\mu_A$ independently of $\xi$, then $\Pi_N^\bot u^N(0)\sim(\Pi_N^\bot)^{-1}\mu_A$ is independent of $\Pi_Nu^N(0)\sim(\Pi_N)^{-1}\mu_A$. Since $\mu_A$ is invariant under $u_{\text{OU}}$, we have $\Pi_N^\bot u^N(t)\sim(\Pi_N^\bot)^{-1}\mu_A$ for all $t$. Thus we need the following lemma to complete the proof of
the invariance of $\mu_A$ under $(u^N)$.

\begin{lemm}\label{galerkin:lemm:U^N does not explode}
If the trilinear condition \eqref{3_eq:1.2-s-hat} holds, then the solution $U^N$ of \eqref{3_eq:finite galerkin} exists globally in time, and admits $\mu_A^N:=(\Pi_N)^{-1}\mu_A$ as an invariant measure.
\end{lemm}

\begin{proof}
Note that if we define $\tilde\Ga_{\a\b\ga}:=(A^{-1})_{\a\a'}\Ga_{\b\ga}^{\a'}$, then the condition \eqref{3_eq:1.2-s-hat} is equivalent to
\begin{align}\label{3_eq:equiv to hatGa}
\tilde\Ga_{\a\b\ga}=\tilde\Ga_{\a\ga\b}=\tilde\Ga_{\b\a\ga}.
\end{align}
Indeed, since $(A^{-1})_{\a\a'}=\sum_{\ga}\tau_\a^\ga\tau_{\a'}^\ga$ and $\tau_{\a'}^\a\Ga_{\b\ga}^{\a'}=\tau_\b^{\b'}\tau_\ga^{\ga'}\hat\Ga_{\b'\ga'}^\a$,
\begin{align*}
\tilde\Ga_{\a\b\ga}=\sum_{\ga'}\tau_\a^{\ga'}\tau_{\a'}^{\ga'}\Ga_{\b\ga}^{\a'}
=\sum_{\ga'}\tau_\a^{\ga'}\tau_\b^{\b'}\tau_\ga^{\ga''}\hat\Ga_{\b'\ga''}^{\ga'}.
\end{align*}
Thus $\hat\Ga_{\b'\ga''}^{\ga'}=\hat\Ga_{\ga'\ga''}^{\b'}$ leads to $\tilde\Ga_{\a\b\ga}=\tilde\Ga_{\b\a\ga}$. 
From \eqref{3_eq:equiv to hatGa}, we have the identity
\begin{align}\label{3_eq:int vani}
(A^{-1})_{\a\b}\langle F_N^\a(u),u^\b\rangle_{L^2(\T)}
&=-\tfrac12(A^{-1})_{\a\b}\Ga_{\b'\ga}^\a\langle P_Nu^{\b'}P_Nu^\ga ,\partial_xP_Nu^\b\rangle_{L^2(\T)}\\
\notag&=-\tfrac16\tilde\Ga_{\b\b'\ga}\int_\T \partial_x(P_Nu^{\b'}P_Nu^\ga P_Nu^\b)(x)dx=0.
\end{align}

We show the global existence of $U^N$. For this, we set
$$
H_t^N:=\sum_\a\|\tau^\a U^N(t)\|_{L^2(\T)}^2\equiv(A^{-1})_{\a\b}\langle U^{N,\a}(t),U^{N,\b}(t)\rangle.
$$
Then, by It\^o's formula, we have
\begin{align*}
dH_t^N
&=2(A^{-1})_{\a\b}\langle U^{N,\a}(t),dU^{N,\b}(t)\rangle
+(A^{-1})_{\a\b}\langle dU^{N,\a}(t),dU^{N,\b}(t)\rangle\\
&=-(A^{-1})_{\a\b}\langle \partial_xU^{N,\a},\partial_x U^{N,\b}\rangle (t) dt
+2(A^{-1})_{\a\b}\langle U^{N,\a},F_N^\b(U^N)\rangle(t)dt\\
&\quad+(A^{-1})_{\a\b}\langle\si_\ga^\a\partial_x\Pi_N\xi^\ga(dt),\si_\delta^\b\partial_x\Pi_N\xi^\delta(dt)\rangle+
2(A^{-1})_{\a\b}\langle U^{N,\a}(t),\si_\ga^\b\partial_x\Pi_N\xi^\ga(dt)\rangle\\
&\le C_Ndt+dM_t^N,
\end{align*}
where $C_N=d\sum_{|k|\le N}4\pi^2k^2$ and $M^N$ is a local martingale with the quadratic variation
\begin{align*}
d[M^N]_t&=4(A^{-1})_{\alpha\beta}\langle\partial_xU^{N,\alpha}(t),\partial_xU^{N,\beta}(t)\rangle dt\\
&=4(A^{-1})_{\a\b}\sum_{|k|\le N}(2\pi ik)\langle U^{N,\a}(t),e^{-2\pi ik\cdot}\rangle\overline{(2\pi ik)\langle U^{N,\b}(t),e^{-2\pi ik\cdot}\rangle}\\
&\le4(2\pi N)^2H_t^Ndt.
\end{align*}
By Doob's inequality, we have
\begin{align*}
E\left[\sup_{0\le t\le T}(M_t^N)^2\right]\le 4E[(M_T^N)^2]\le 64\pi^2N^2\int_0^TE(H_t^N)dt.
\end{align*}
Thus, we obtain
\begin{align*}
E\left[\sup_{0\le t\le T}H_t^N\right]&\le H_0^N+C_NT+E\left[\sup_{0\le t\le T}|M_t^N|\right]\\
&\le H_0^N+C_NT+1+E\left[\sup_{0\le t\le T}(M_t^N)^2\right]\\
&\le H_0^N+C_NT+1+64\pi^2N^2\int_0^TE(H_t^N)dt.
\end{align*}
Applying Gronwall's inequality, we obtain
\begin{align*}
E\left[\sup_{0\le t\le T}H_t^N\right]\le(H_0^N+C_NT+1)e^{64\pi^2N^2T}<\infty.
\end{align*}
This implies that the process $(U^N(t))_{0\le t<\infty}$ does not explode because $A^{-1}$ is non-degenerate.

Next we show the invariance of $\mu_A^N$. For the sake of simplicity, we consider the orthonormal basis
$$
e_k(x)=
\begin{cases}
\sqrt{2}\cos(2\pi kx),&k>0,\\
\sqrt{2}\sin(2\pi kx),&k<0,
\end{cases}
$$
of $H=\{f\in L^2(\T)\,;\int_\T f(x)dx=0\}$. We write $u^{\a,k}=\langle u^\a,e_k\rangle$ for $u=(u^\a)_{\a=1}^d\in H^d$. For every smooth function $\Phi:\R^{2Nd}(\simeq(\Pi_NH)^d)\to\R$, which has bounded derivatives, by It\^o's formula,
\begin{align*}
d\Phi(U^N)(t)&=\partial_{(\a,k)}\Phi(U^N)dU^{N,\a,k}(t)+\tfrac12\partial_{(\a,k)(\b,k)}^2\Phi(U^N)dU^{N,\a,k}dU^{N,\b,l}(t)\\
&=-\sum_{|k|\le N}2\pi^2k^2(U^{N,\a,k}\partial_{(\a,k)}-A^{\a\b}\partial_{(\a,k)(\b,k)}^2)\Phi(U^N)dt\\
&\quad+F_N^{\a,k}(U^N)\partial_{(\a,k)}\Phi(U^N)dt+\text{martingale}\\
&=:L_1^N\Phi(U^N)dt+L_2^N\Phi(U^N)dt+\text{martingale}.
\end{align*}
By Echeverr{\'i}a's criterion \cite{Echeverria}, to complete the proof of the invariance of $\mu_A^N$,
it suffices to show $\int(L_1^N+L_2^N)\Phi(u)\mu_A^N(du)=0$. Since $L_1$ is the generator of $\Pi_Nu_{\text{OU}}$, under which $\mu_A^N$ is invariant, we have
$$
\int L_1^N\Phi(u)\mu_A^N(du)=0.
$$
For $L_2$, note that under $\mu_A^N$ the $\R^d$-valued random variables $\{(u^{\a,k})\}_k$ are
independent and each of them has the distribution $\mathcal{N}(0,A)$. Since $\mathcal{N}(0,A)$ has a density function
$\ga_A(u^k), u^k\in \R^d,$ satisfying $\partial_{\b}\ga_A(u^k)=-(A^{-1})_{\a\b}u^{\a,k}\ga_A(u^k)$, by the integration by parts,
we have
\begin{align*}
\int L_2^N\Phi(u)\mu_A^N(du)
&=\sum_k\int_{\R^{2Nd}}\partial_{(\a,k)}\Phi(u)F_N^{\a,k}(u)\prod_{0<|l|\le N}\ga_A(u^l)du^l \\
&=\sum_k\int_{\R^{2Nd}}\Phi(u)F_N^{\a,k}(u)(A^{-1})_{\a\b}u^{\b,k}\prod_{0<|l|\le N}\ga_A(u^l)du^l\\
&=\int_{(\Pi_NH)^d}\Phi(u)(A^{-1})_{\a\b}\langle F_N^\a(u),u^\b\rangle_{L^2(\T)}\mu_A^N(du)=0
\end{align*}
from \eqref{3_eq:int vani}. Here we have used $\sum_k\partial_{(\a,k)}F_N^{\a,k}(u)=0$, which is shown as follows:
\begin{align*}
\sum_k\partial_{(\a,k)}F_N^{\a,k}(u)
&=\tfrac12\sum_k\langle\partial_{(\a,k)}\partial_x\Ga_{\b\ga}^\a P_N(P_Nu^\b P_Nu^\ga),e_k\rangle\\
&=\sum_k\langle\partial_x\Ga_{\a\b}^\a P_N (P_Nu^\b P_Ne_k),e_k\rangle\\
&=-\tfrac12\Ga_{\a\b}^\a\langle P_Nu^\b,\partial_x\sum_k(P_Ne_k)^2\rangle=0,
\end{align*}
since $\sum_k(P_Ne_k)^2=\sum_{k>0}\psi(N^{-1}k)^2(e_k^2+e_{-k}^2)=2\sum_{k>0}\psi(N^{-1}k)^2$ is a constant, so that it vanishes after applying $\partial_x$.
\end{proof}

If $\Pi_N^\bot u^N(t)\sim(\Pi_N^\bot)^{-1}\mu_A$ and $\Pi_Nu^N(t)\sim\mu_A^N$, then $u^N(t)\sim\mu_A$ by definition. As a consequence, we have the invariance of $\mu_A$ under $(u^N)$.

\begin{prop}\label{3_prop:muA is inv under uN}
If the trilinear condition \eqref{3_eq:1.2-s-hat} holds, the solution $u^N$ of \eqref{3_galerkin:stationary} exists globally in time, and admits $\mu_A$ as an invariant measure.
\end{prop}

\subsection{Global existence for a.e.-initial values}\label{subsection:global existence for ae initial value}

Let $\mathbb{U},\mathbb{U}^N$ be the $\mathcal{U}_{\text{CSB}}^\k$-random variables defined from the space-time white noise $\xi$, corresponding to the equations \eqref{3_eq:sbe} and \eqref{3_galerkin:stationary}, respectively. The following result is obtained similarly to Theorem \ref{3_conv of driv} and the proof is omitted.

\begin{theo}
For every $T>0$ and $p\ge1$,
$$
\lim_{N\to\infty}E\$\mathbb{U}^N-\mathbb{U}\$_T^p=0.
$$
In particular, $u^N=S_{\text{\rm CSB}}^N(v(0),w(0),\mathbb{U}^N)$ converges to $u=S_{\text{\rm CSB}}(v(0),w(0),\mathbb{U})$ in probability as $N\to\infty$ in $C([0,T_{\text{\rm sur}}),(\mathcal{C}_0^{(\k-1)\wedge\mu\wedge2\mu})^d)$.
\end{theo}

From now we set $\mu=\tfrac{\k-1}2$, so that $(\k-1)\wedge\mu\wedge2\mu=\k-1$. We consider initial values $u(0)\in (\mathcal{C}_0^{\k-1})^d$ and set
\begin{align}\label{3_eq:ini of v w}
v(0)=0,\quad
w(0)=u(0)-U_\I(0)-U_\Y(0)-U_\W(0).
\end{align}
We can prove the following result in a similar way to Theorem 5.1 of \cite{DPD}. Note that if $T_{\text{sur}}$ is finite then
\begin{align}\label{3_eq:u explodes}
\lim_{t\uparrow T_{\text{sur}}}\|u\|_{C([0,t],(\mathcal{C}_0^{\k-1})^d)}=\infty,
\end{align}
cf., Theorem \ref{3_thm:3}-(2).  Our main result of this section is formulated as follows.

\begin{theo}  \label{3_thm:5.7}
We assume the trilinear condition \eqref{3_eq:1.2-s-hat}.  Then, for every $T>0$ and $\mu_A$-a.e.\ $u(0)\in(\mathcal{C}_0^{\k-1})^d$, 
there exists a unique solution $(v,w)\in\mathcal{D}_{\text{\rm CSB}}^{\la,\mu}([0,T])$ of the system \eqref{3_sys:sbe} with initial values \eqref{3_eq:ini of v w}, which satisfies for every $p\ge1$,
\[E\|S_{\text{\rm CSB}}(v(0),w(0),\mathbb{U})\|_{C([0,T],(\mathcal{C}_0^{\k-1})^d)}^p<\infty.\]
In particular, $T_{\text{\rm sur}}=\infty$ a.s. Furthermore, $u=S_{\text{\rm CSB}}(v(0),w(0),\mathbb{U})$ is a Markov process on $(\mathcal{C}_0^{\k-1})^d$ which admits $\mu_A$ as an invariant measure.
\end{theo}

By the equivalence of $S_{\text{\rm KPZ}}$ and $S_{\text{\rm CSB}}$ shown in Theorem \ref{3_kpz equiv csb}, 
we have $T_{\text{\rm sur}}^{\text{\rm KPZ}}=\infty$ a.s. when $\partial_xh(0)$ is taken from $\mu_A$-full set 
of $(\mathcal{C}_0^{\k-1})^d$.  This together with Theorem \ref{3_thm:5.7} implies Theorem \ref{3_thm:1.3}.

\begin{proof}[Proof of Theorem \ref{3_thm:5.7}]
We denote by $u^N(\cdot,u(0))$ the solution of (\ref{3_galerkin:stationary}) with initial value $u(0)$. The local existence result like Theorem \ref{3_thm:3}-(1) holds even for the stochastic Burgers equation, which implies that there exist $C,C'>0$ independent of $N$ such that, for given $M>0$, if
\[\$\mathbb{U}^N\$_T^3\le CM,\quad
\|u(0)\|_{(\mathcal{C}_0^{\k-1})^d}\le CM,\]
then
\[\sup_{t\in[0,t_M^*]}\|u^N(t,u(0))\|_{(\mathcal{C}_0^{\k-1})^d}\le M,
\quad t_M^*=C'M^{-\frac{2}{\k-\la}}\wedge T.\]
This yields that for every $u(0)\in(\mathcal{C}_0^{\k-1})^d$,
\begin{align*}
&P(\sup_{t\in[0,T]}\|u^N(t,u(0))\|_{(\mathcal{C}_0^{\k-1})^d}>M)\\
&\le\sum_{k=0}^{[T/t_M^*]}P(\sup_{t\in[kt_M^*,(k+1)t_M^*]}\|u^N(t,u(0))\|_{(\mathcal{C}_0^{\k-1})^d}>M)\\
&\le\sum_{k=0}^{[T/t_M^*]}P(\|u(kt_M^*,u(0))\|_{(\mathcal{C}_0^{\k-1})^d}>CM)
+([T/t_M^*]+1)P(\$\mathbb{U}^N\$_{2T}^3>CM).
\end{align*}
Since $\sup_NE[\$\mathbb{U}^N\$_{2T}^p]<\infty$ for every $p\ge1$, the second term is bounded as
\[([T/t_M^*]+1)P(\$\mathbb{U}^N\$_{2T}^3>CM)\lesssim_p M^{-p}.\]
For the first term, since $\mu_A$ is invariant under $u^N$ and $\mu_A$ is Gaussian, we have
\begin{align*}
&\int_{(\mathcal{C}_0^{\k-1})^d}P(\|u(kt_M^*,u(0))\|_{(\mathcal{C}_0^{\k-1})^d}>CM)\mu_A(du(0))\\
&=\mu_A(\|u(0)\|_{(\mathcal{C}_0^{\k-1})^d}>CM)\lesssim_p M^{-p}.
\end{align*}
These are summarized into the bound
\[\int_{(\mathcal{C}_0^{\k-1})^d}P(\sup_{t\in[0,T]}\|u^N(t,u(0))\|_{(\mathcal{C}_0^{\k-1})^d}>M)\mu_A(du(0))\lesssim_pM^{-p},\]
for every $p\ge1$, which leads
\[\int_{(\mathcal{C}_0^{\k-1})^d}E[\sup_{t\in[0,T]}\|u^N(t,u(0))\|_{(\mathcal{C}_0^{\k-1})^d}^p]\mu_A(du(0))\lesssim_p1.\]
Since bounded set in $L^p(\Omega\times(\mathcal{C}_0^{\k-1})^d,P\times\mu_A)$ is weak$^*$ compact for $p>1$, there exists a subsequence $\{N_k\}$ and $M\in L^p(\Omega\times(\mathcal{C}_0^{\k-1})^d,P\times\mu_A)$ such that
\[\sup_{t\in[0,T]}\|u^{N_k}(t,u(0))\|_{(\mathcal{C}_0^{\k-1})^d}\to M\]
as $k\to\infty$ in weak$^*$ topology. 
On the other hand, for every $u(0)\in(\mathcal{C}_0^{\k-1})^d$ and $m\in\mathbb{N}$, $u^{N_k}$ converges to $u=S_{\text{CSB}}(v(0),w(0),\mathbb{U})$ in probability in the space $C([0,(\frac{m-1}{m} T_{\text{sur}})\wedge T],(\mathcal{C}_0^{\k-1})^d)$. Since $\|u^{N_k}(\cdot,u(0))\|_{C([0,(\frac{m-1}{m} T_{\text{sur}})\wedge T],(\mathcal{C}_0^{\k-1})^d)}$ is also bounded in $L^p(\Omega\times(\mathcal{C}_0^{\k-1})^d,P\times\mu_A)$, there exists a subsequence $\{N_{k_m}\}$ and $M^m\in L^p(\Omega\times(\mathcal{C}_0^{\k-1})^d,P\times\mu_A)$ such that
\[\sup_{t\in[0,(\frac{m-1}{m} T_{\text{sur}})\wedge T]}\|u^{N_{k_m}}(\cdot,u(0))\|_{(\mathcal{C}_0^{\k-1})^d}\to M^m\]
as $k_m\to\infty$ in weak$^*$ topology for every $m\in \N$. From the uniqueness of the limit, we have
\[\|u\|_{C([0,(\frac{m-1}{m} T_{\text{sur}})\wedge T],(\mathcal{C}_0^{\k-1})^d)}=M^m\le M,\quad
\text{$P\times\mu_A$-a.e.}\]
Taking the limit $m\to\infty$, we have
\[\|u\|_{C([0,T_{\text{sur}}\wedge T),(\mathcal{C}_0^{\k-1})^d)}\le M,\quad
\text{$P\times\mu_A$-a.e.}\]
This implies that for $\mu_A$-a.e.\ $u(0)$, $u(t)$ does not explode as $t\uparrow T_{\text{sur}}$ against \eqref{3_eq:u explodes},
so that $T_{\text{sur}}\ge T$ a.s. Since $T$ is arbitrary, $T_{\text{sur}}=\infty$ a.s.

Markov property of $u$ is obtained as follows. Let $u=u(\cdot,u(0),\mathbb{U})\in C([0,T_{\text{sur}}),(\mathcal{C}_0^{\k-1})^d)$ be the solution starting at $u(0)$ and driven by $\mathbb{U}$. Now we introduce the space $(\mathcal{C}_0^{\k-1})^d\cup\{\Delta\}$ and the distance
$$
d(u,\Delta)=(1+\|u\|_{(\mathcal{C}_0^{\k-1})^d})^{-1},
$$
and, by defining $u(t)=\Delta$ for $t\ge T_{\text{sur}}$, we regard $u$ as a random variable taking values in $C([0,\infty),(\mathcal{C}_0^{\k-1})^d\cup\{\Delta\})$. By the uniqueness result, we have the identity
\[u(t,u(s,u(0),\mathbb{U}),\tau_s\mathbb{U})=u(s+t,u(0),\mathbb{U}),\quad s,t\ge0,\]
where $\tau_s$ is a shift operator defined by $\tau_s\mathbb{U}(\cdot)=\mathbb{U}(\cdot+s)$. 
Let $\mathcal{F}_t$ be the $\si$-algebra generated by $\{\xi(\mathbf{1}_{(r,s]}\cdot)\}_{-\infty<r<s\le t}$. 
Then from the construction of the solution, for deterministic element $v$, $u(s,v,\mathbb{U})$ is $\mathcal{F}_s$-measurable. 
Moreover, $u(t,v,\tau_s\mathbb{U})$ is independent of $\mathcal{F}_s$ because it coincides with the solution of (the approximation of)
\begin{align*}
\begin{cases}
\partial_tu^\alpha=\tfrac12\partial_x^2u+\tfrac12\Gamma_{\beta\gamma}^\alpha\partial_x(u^\beta u^\gamma)+\sigma_\beta^\alpha\partial_x(\tau_s\xi^\beta)&t>0,\\
u(0,\cdot)=v.
\end{cases}
\end{align*}
As a consequence, for every bounded measurable function $\Phi$ on $(\mathcal{C}_0^{\k-1})^d\cup\{\Delta\}$,
we have
\[E[\Phi(u(s+t,u(0),\mathbb{U}))|\mathcal{F}_s]=P_{st}(u(s,u(0),\mathbb{U})),\]
where $P_{st}(v):=E[\Phi(u(t,v,\tau_s\mathbb{U}))]$. In particular,  when $u(0)\sim\mu_A$ independently of $\xi$,
since $T_{\text{sur}}=\infty$ almost surely, $u$ is a Markov process on $(\mathcal{C}_0^{\k-1})^d$. 
The invariance of $\mu_A$ under $(u(t))_{t\ge0}$ is an immediate consequence of the convergence $u^N(t)\to u(t)$
and Proposition \ref{3_prop:muA is inv under uN}.
\end{proof}

\subsection{Global existence for two coupled KPZ approximating equations}
\label{section 6}

The derivatives $u=\partial_xh^\e$ and
$\tilde u=\partial_x\tilde h^\e$ of the solutions $h^\e$
and $\tilde h^\e$ of the coupled KPZ approximating equations
\eqref{3_eq:1-2} and \eqref{3_eq:1} solve the following coupled stochastic Burgers equations
\begin{align}\label{eq:6.1}
\partial_tu^\a=\tfrac12\partial_x^2u^\a
+\tfrac12\Ga_{\b\ga}^\a\partial_x(u^\b u^\ga)
+\si_\b^\a\partial_x\xi^\b*\eta^\e,
\end{align}
and
\begin{align}\label{eq:6.2}
\partial_t \tilde u^\a=\tfrac12\partial_x^2 \tilde u^\a
+\tfrac12\Ga_{\b\ga}^\a\partial_x\{( \tilde u^\b  \tilde u^\ga)*\eta_2^\e\}
+\si_\b^\a\partial_x\xi^\b*\eta^\e,
\end{align}
respectively. We show the global existence of the solutions of these two equations. Since $\e>0$ is fixed, we drop the
superscripts $\e$ from $u^\e$ and $\tilde u^\e$ and
simply write $u$ and $\tilde u$, respectively.

\subsubsection{The equation \eqref{eq:6.1}}

For the equation \eqref{eq:6.1}, we apply the classical method of the energy inequality
to show the global well-posedness.

\begin{lemm}
Assume $E\big[\|u(0)\|_{L^2(\T,\R^d)}^2\big] < \infty$. Then, for every $T>0$, we have
$$
E\left[ \sup_{0\le t \le T} \|u(t)\|_{L^2(\T,\R^d)}^2\right] < \infty.
$$
In particular, $T_{\rm sur}=\infty$ a.s.
\end{lemm}

\begin{proof}
We use a similar argument given in Section \ref{subsection:galerkin approximation}. We consider the approximation
\begin{align*}
\begin{cases}
\partial_tu^{N,\alpha}=\tfrac{1}{2}\partial_x^2u^{N,\alpha}+F_N^\alpha(u^N)+\sigma_\beta^\alpha\partial_x\Pi_N\xi^\beta*\eta^\epsilon,\\
u^N(0)=\Pi_Nu(0).
\end{cases}
\end{align*}
where $F_N^\alpha$ is the operator defined by \eqref{galerkin:def of F^N}. We can use the same argument as in the proof of Lemma \ref{galerkin:lemm:U^N does not explode} and show the global existence of $(u^N(t))_{0\le t<\infty}$. Indeed, by applying It\^o's formula to $H^N:=\sum_\a\|\tau^\a u^N\|_{L^2(\T)}^2$ we have
\begin{align*}
dH_t^N\le C_Ndt+dM_t^N,
\end{align*}
where $C_N=d\sum_{|k|\le N}4\pi^2k^2\varphi(\epsilon k)^2$ and $M^N$ is a local martingale with the quadratic variation
\begin{align*}
d[M^N]_t&=4(A^{-1})_{\a\b}\langle u^{N,\a}*\partial_x\eta^\epsilon(t),u^{N,\b}*\partial_x\eta^\epsilon(t)\rangle dt\\
&=4\sum_\a\|\tau^\a u^N*\partial_x\eta^\epsilon(t)\|_{L^2(\T)}^2dt\le 4C'H_t^Ndt,
\end{align*}
where $C'=\|\partial_x\eta^\epsilon\|_{L^1(\T)}^2$. The same argument as Lemma \ref{galerkin:lemm:U^N does not explode} shows that
\begin{align*}
E\left[\sup_{0\le t\le T}\sum_\a\|\tau^\a u^N(t)\|_{L^2(\T)}^2\right]\le (E[\sup_\a\|\tau^\a \Pi_Nu(0)\|_{L^2(\T)}^2]+C_NT+1)e^{16C'T}.
\end{align*}
Since $\tau$ is invertible, a similar estimate holds for $\sum_\a\|u^{N,\a}(t)\|_{L^2(\T)}^2=\|u^N(t)\|_{L^2(\T,\R^d)}^2$. Since $\|\Pi_Nu(0)\|_{L^2(\T,\R^d)}\le\|u(0)\|_{L^2(\T,\R^d)}$ and $C_N\le d\sum_k4\pi^2k^2\varphi(\epsilon k)^2<\infty$, we have the uniform estimate
\begin{align*}
\sup_NE\left[\sup_{0\le t\le T}\|u^N(t)\|_{L^2(\T,\R^d)}^2\right]<\infty.
\end{align*}
It is not difficult to prove that $(u^N)_N$ converges to the solution $u$ of \eqref{eq:6.1} up to time $T_{\text{sur}}$, as an application of the paracontrolled calculus. From this convergence, in a similar way to the proof of Theorem \ref{3_thm:5.7}, we can show that
$$
E[\|u\|_{C([0,T_{\text{sur}}\wedge T),L^2(\T,\R^d))}^2]<\infty
$$
and this implies $T_{\text{sur}}=\infty$ a.s.
\end{proof}

Next we consider the case that $u_0\in(\mathcal{C}_0^{\delta-1})^d$. We fix $T>0$. By Theorem \ref{3_thm:3}, for every $K>0$ there exists a (deterministic) $t=t(u_0,K)\in(0,T]$ such that
$$
u_t^K=
\begin{cases}
u_t,&\$\mathbb{H}^\e\$_t\le K,\\
0,&\text{otherwise}
\end{cases}
$$
satisfies $\|u_t^K\|_{L^2(\T,\R^d)}\lesssim1+\|u_0\|_{(\mathcal{C}_0^{\delta-1})^d}+K^3$, so that $E\|u_t^K\|_{L^2(\T,\R^d)}^2<\infty$. Since the solution of \eqref{eq:6.1} with initial value $u_t^K$ exists globally, we have
$$
P(\text{$u\in C([0,T],(\mathcal{C}_0^{\delta-1})^d)$})\ge P(\$\mathbb{H}^\e\$_t\le K)\ge P(\$\mathbb{H}^\e\$_T\le K).
$$
By letting $K\to\infty$, we see that $u$ exists up to the time $T$ almost surely. Since $T>0$ is arbitrary, we have the global existence of $u$.

\subsubsection{The equation \eqref{eq:6.2}}

For the equation \eqref{eq:6.2}, in a similar way to Sections \ref{subsection:galerkin approximation}-\ref{subsection:global existence for ae initial value}, one can first show the global existence for a.e.-initial values under the stationary measure, and extend it to all initial values by combining it with the strong Feller property.

Precisely, we first use the approximation
\begin{align*}
\begin{cases}
\partial_t\tilde{u}^{N,\alpha}=\tfrac{1}{2}\partial_x^2\tilde{u}^{N,\alpha}+F_N^\alpha(\tilde{u}^N)*\eta_2^\epsilon+\sigma_\beta^\alpha\partial_x\xi^\beta*\eta^\epsilon,\\
\tilde{u}^N(0)=\tilde{u}(0)\in(\mathcal{C}_0^{\delta-1})^d.
\end{cases}
\end{align*}
where $F_N^\alpha$ is the operator defined by \eqref{galerkin:def of F^N}. Without loss of generality, we may assume $\varphi(\epsilon k)\neq0$ for all $k\in\Z$. We can see that this equation has a unique global solution $\tilde{u}^N$ by applying It\^o's formula to
$$
H^N:=\sum_\a\|\tau^\a\varphi^{-1}(\e D)\Pi_N\tilde{u}^N\|_{L^2(\T)}^2\equiv(A^{-1})_{\a\b}\langle \Pi_N\tilde{u}^{N,\a},\varphi^{-2}(\epsilon D)\Pi_N\tilde{u}^{N,\b}\rangle.
$$
Moreover, $\tilde{u}^N$ admits the measure $\mu_A^\epsilon$ as an invariant measure, where $\mu_A^\epsilon$ is the distribution on $\mathcal{D}'(\T,\R^d)$ of $u^\alpha(x)=\sum_k\varphi(\epsilon k)u^{\alpha,k}e^{2\pi ikx}$ with the family $\{u^{\alpha,k}\}$ of centered complex Gaussian variables such that $u^{\alpha,-k}=\overline{u^{\alpha,k}}$ and covariance
$$
E[u^{\a,k}u^{\b,l}]=A^{\a\b}1_{\{k+l=0\}},
$$
recall Lemma \ref{galerkin:lemm:U^N does not explode}. Then the similar argument to the proof of Theorem \ref{3_thm:5.7} shows that the solution $\tilde{u}$ of the equation \eqref{eq:6.2} exists globally for $\mu_A^\epsilon$-a.e. initial values.

\section*{Acknowledgements}
The authors thank C\'edric Bernardin for bringing the paper \cite{EK} into their attention,
Xue-Mei Li for her interest in this paper, especially, the discussion on the trilinear conditions
\eqref{3_eq:1.2-s} and \eqref{3_eq:1.2-s-hat}, and Hao Shen for his significant remarks on the logarithmic renormalization factors. They also thank anonymous referees for their helpful remarks, in particular, suggestions of the condition \eqref{true condition for no log} and Remarks \ref{rem:section3 graph and F} and \ref{rem:section3 graph and G}.

\end{document}